\documentclass[11pt,reqno]{amsart}
\usepackage{amsmath}
\usepackage{amsthm}
\usepackage{graphicx}
\usepackage{amsfonts}
\usepackage{amssymb}
\usepackage{color}
\usepackage{enumerate}
\usepackage[margin=1.1in]{geometry}
\numberwithin{equation}{section}

\newtheorem{theorem}{Theorem}[section]
\newtheorem{lemma}[theorem]{Lemma}
\newtheorem{proposition}[theorem]{Proposition}
\newtheorem{corollary}[theorem]{Corollary}

\theoremstyle{definition}

\newtheorem{definition}[theorem]{Definition}
\newtheorem{remark}[theorem]{Remark}

\newtheorem*{assumption}{Standing assumptions}

\def\E{{\mathbb E}}

\def\R{{\mathbb R}}

\def\PP{{\mathbb P}}

\def\P{{\mathcal P}}

\def\X{{\mathcal X}}
\def\Y{{\mathcal Y}}

\def\W{{\mathcal W}}

\def\F{{\mathcal F}}

\title{A non-exponential extension of Sanov's theorem via convex duality}

\author{Daniel Lacker}
\address{Industrial Engineering \& Operations Research, Columbia University}
\email{daniel.lacker@columbia.edu}
\thanks{This material is based upon work supported by the National Science Foundation under Award No. DMS 1502980.}

\begin{document}

\begin{abstract}
This work is devoted to a vast extension of Sanov's theorem, in Laplace principle form, based on alternatives to the classical convex dual pair of relative entropy and cumulant generating functional. The abstract results give rise to a number of probabilistic limit theorems and asymptotics. For instance, widely applicable non-exponential large deviation upper bounds are derived for empirical distributions and averages of i.i.d.\ samples under minimal integrability assumptions, notably accommodating heavy-tailed distributions. Other interesting manifestations of the abstract results include new results on the rate of convergence of empirical measures in Wasserstein distance, uniform large deviation bounds, and variational problems involving optimal transport costs, as well as an application to error estimates for approximate solutions of stochastic optimization problems. The proofs build on the Dupuis-Ellis weak convergence approach to large deviations as well as the duality theory for convex risk measures.
\\ {\ } 
\end{abstract}

\maketitle

\section{Introduction}

An original goal of this paper was to extend the weak convergence methodology of Dupuis and Ellis \cite{dupuis-ellis} to the context of non-exponential (e.g., heavy-tailed) large deviations. While we claim only modest success in this regard, we do find some general-purpose large deviation upper bounds which can be seen as polynomial-rate analogs of the upper bounds in the classical theorems of Sanov and Cram\'er. At least as interesting, however, are the abstract principles behind these bounds, which have broad implications beyond the realm of large deviations.
Let us first describe these abstract principles before specializing them in various ways.

Let $E$ be a Polish space, and let $\P(E)$ denote the set of Borel probability measures on $E$ endowed with the topology of weak convergence. Let $B(E)$ (resp. $C_b(E)$) denote the set of measurable (resp. continuous) and bounded real-valued functions on $E$. For $n \ge 1$ and $\nu \in \P(E^n)$, define $\nu_{0,1} \in \P(E)$ and measurable maps $\nu_{k-1,k} : E^{k-1} \rightarrow \P(E)$ for $k=2,\ldots,n$ via the disintegration
\[
\nu(dx_1,\ldots,dx_n) = \nu_{0,1}(dx_1)\prod_{k=2}^n\nu_{k-1,k}(x_1,\ldots,x_{k-1})(dx_k).
\]
In other words, if $(X_1,\ldots,X_n)$ is an $E^n$-valued random variable with law 
$\nu$, then $\nu_{0,1}$ is the law of $X_1$, and $\nu_{k-1,k}(X_1,\ldots,X_{k-1})$ is the conditional law of $X_k$ given $(X_1,\ldots,X_{k-1})$. Of course, $\nu_{k-1,k}$ are uniquely defined up to $\nu$-almost sure equality.

The protagonist of the paper is a proper (i.e., not identically $\infty$) convex function $\alpha : \P(E) \rightarrow (-\infty,\infty]$ with compact sub-level sets; that is, $\{\nu \in \P(E) : \alpha(\nu) \le c\}$ is compact for every $c \in \R$. For $n \ge 1$ define
$\alpha_n : \P(E^n) \rightarrow (-\infty,\infty]$ by
\[
\alpha_n(\nu) = \int_{E^n}\sum_{k=1}^n\alpha(\nu_{k-1,k}(x_1,\ldots,x_{k-1}))\,\nu(dx_1,\ldots,dx_n),
\]
and note that $\alpha_1 \equiv \alpha$.
Define the convex conjugate $\rho_n : B(E^n) \rightarrow \R$ by
\begin{align}
\rho_n(f) = \sup_{\nu \in \P(E^n)}\left(\int_{E^n}f\,d\nu - \alpha_n(\nu)\right), \quad\quad \text{ and } \quad\quad \rho \equiv \rho_1. \label{intro:duality}
\end{align}
Our main interest is in evaluating $\rho_n$ at functions of the \emph{empirical measure} $L_n : E^n \rightarrow \P(E)$ defined by
\[
L_n(x_1,\ldots,x_n) = \frac{1}{n}\sum_{i=1}^n\delta_{x_i}.
\]
The main abstract result of the paper is the following extension of Sanov's theorem, proven in more generality in Section \ref{se:sanovproof} by adapting the weak convergence techniques of Dupuis-Ellis \cite{dupuis-ellis}.

\begin{theorem} \label{th:main-sanov}
For $F \in C_b(\P(E))$,
\[
\lim_{n\rightarrow\infty}\frac{1}{n}\rho_n(nF \circ L_n) = \sup_{\nu \in \P(E)}(F(\nu) - \alpha(\nu)).
\]
\end{theorem}

The guiding example is the relative entropy, $\alpha(\cdot) = H(\cdot | \mu)$, where $\mu \in \P(E)$ is a fixed reference measure, and $H$ is defined by
\begin{align}
H(\nu | \mu) = \int_E\log(d\nu/d\mu)\,d\nu, \text{ for } \nu \ll \mu, \quad\quad H(\nu | \mu) = \infty \text{ otherwise}, \label{def:relativeentropy}
\end{align}
Letting $\mu^n$ denote the $n$-fold product measure, it turns out that $\alpha_n(\cdot) = H(\cdot | \mu^n)$, by the  so-called \emph{chain rule} of relative entropy \cite[Theorem B.2.1]{dupuis-ellis}.
The dual $\rho_n$ is well known to be $\rho_n(f) = \log\int_{E^n}e^f\,d\mu^n$, and the duality formulas relating $\rho_n$ and $\alpha_n$ are often known as the Gibbs variational principle or the Donsker-Varadhan formula \cite[Proposition 1.4.2 and Lemma 1.4.3]{dupuis-ellis}. In this case Theorem \ref{th:main-sanov} reduces to the Laplace principle form of Sanov's theorem:
\[
\lim_{n\rightarrow\infty}\frac{1}{n}\log\int_{E^n}e^{nF\circ L_n}\,d\mu^n = \sup_{\nu \in \P(E)}(F(\nu) - H(\nu | \mu)).
\]
Well known theorems of Varadhan and Dupuis-Ellis (see \cite[Theorem 1.2.1 and 1.2.3]{dupuis-ellis}) assert the equivalence of this form of Sanov's theorem with the more common form: for every Borel set $A \subset \P(E)$ with closure $\overline{A}$ and interior $A^\circ$,
\begin{align}
-\inf_{\nu \in A^\circ}H(\nu | \mu) &\le \liminf_{n\rightarrow\infty}\frac{1}{n}\log\mu^n(L_n \in A) \nonumber \\
	&\le \limsup_{n\rightarrow\infty}\frac{1}{n}\log\mu^n(L_n \in A) \le -\inf_{\nu \in \overline{A}}H(\nu | \mu). \label{def:classicalsanov}
\end{align}
To derive this heuristically, apply Theorem \ref{th:main-sanov} to the function
\begin{align}
F(\nu) = \begin{cases}
0 &\text{if } \nu \in A \\
-\infty &\text{otherwise}.
\end{cases} \label{def:convexindicator}
\end{align}
For general $\alpha$, Theorem \ref{th:main-sanov} does not permit an analogous \emph{equivalent} formulation in terms of deviation probabilities. In fact, for many $\alpha$, Theorem \ref{th:main-sanov} has nothing to do with large deviations (see Sections \ref{se:intro:lln} and \ref{se:intro:optimaltransport} below). Nonetheless, for certain $\alpha$, Theorem 1.1 \emph{implies} interesting large deviations upper bounds, which we prove by formalizing the aforementioned heuristic.
While many $\alpha$ admit fairly explicit known formulas for the dual $\rho$, the recurring challenge in applying Theorem \ref{th:main-sanov} is finding a useful expression for $\rho_n$, and herein lies but one of many instances of the wonderful tractability of relative entropy.
The examples to follow do admit good expressions for $\rho_n$, or at least workable one-sided bounds, but we also catalog in Section \ref{se:intro:alternatives} some natural alternative choices of $\alpha$ for which we did not find useful bounds or expressions for $\rho_n$.

The functional $\rho$ is (up to a sign change) a \emph{convex risk measure}, in the language of F\"ollmer and Schied \cite{follmer-schied-book}.
A rich duality theory for convex risk measures emerged over the past two decades, primarily geared toward applications in financial mathematics and optimization. We take advantage of this theory in Section \ref{se:riskmeasures} to demonstrate how $\alpha$ can be reconstructed from $\rho$, which shows that $\rho$ could be taken as the starting point instead of $\alpha$. Additionally, the theory of risk measures provides insight on how to deal with the subtleties that arise in extending the domain of $\rho$ (and Theorem \ref{th:main-sanov}) to accommodate unbounded functions or stronger topologies on $\P(E)$. Section \ref{se:intro:riskmeasures} briefly reinterprets Theorem \ref{th:main-sanov} in a language more consistent with the risk measure literature. The reader familiar with risk measures may notice a \emph{time consistent dynamic risk measure} (see \cite{acciaio-penner-dynamic} for definitions and survey) hidden in the definition of $\rho_n$ above.

We will make no use of the interpretation in terms of dynamic risk measures, but it did inspire a recursive formula for $\rho_n$ (similar to a result of \cite{cheridito2011composition}). To state it loosely, if $f \in B(E^n)$ then we may write 
\begin{align}
\rho_n(f) = \rho_{n-1}(g), \quad \text{where} \quad g(x_1,\ldots,x_{n-1}) := \rho\left(f(x_1,\ldots,x_{n-1},\cdot)\right). \label{def:intro:rhon-recursive}
\end{align}
To make rigorous sense of this, we must note that $g : E^{n-1} \rightarrow \R$ is merely upper semianalytic and not Borel measurable in general and argue that  $\rho$ is well defined for such functions. We make this precise in Proposition \ref{pr:rhon-iterative}. This recursive formula is not essential for any of them main arguments but is convenient for some calculations.

\subsection{Nonexponential large deviations} \label{se:intro:nonexpLDP}

Our first application, and the one we discuss in the most detail, comes from applying (an extension of) Theorem \ref{th:main-sanov} with
\begin{align}
\alpha(\nu) = \|d\nu/d\mu\|_{L^p(\mu)}-1, \text{ for } \nu \ll \mu, \quad\quad \alpha(\nu) = \infty \text{ otherwise}, \label{intro:lpentropy}
\end{align}
where $\mu \in \P(E)$ is fixed. We state the abstract result first. For a continuous function $\psi : E \rightarrow \R_+ := [0,\infty)$, let $\P_\psi(E)$ denote the set of $\nu \in \P(E)$ satisfying $\int\psi\,d\nu < \infty$. Equip $\P_\psi(E)$ with the topology induced by the linear maps $\nu \mapsto \int f\,d\nu$, where $f : E \rightarrow \R$ is continuous and $|f| \le 1+\psi$.
Recall in the following that $\mu^n$ denotes the $n$-fold product measure.

\begin{theorem} \label{th:lpsanov}
Let $q \in (1,\infty)$, and let $p=q/(q-1)$ denote the conjugate exponent. Let $\mu \in \P(E)$, and suppose $\int\psi^q\,d\mu < \infty$ for some continuous $\psi : E \rightarrow \R_+$. Then, for every closed set $A \subset \P_\psi(E)$,
\[
\limsup_{n\rightarrow\infty}\,n^{q-1}\mu^n(L_n \in A)  \le \left(\inf_{\nu \in A}\|d\nu/d\mu\|_{L^p(\mu)}-1\right)^{-q}.
\]
\end{theorem}

We view Theorem \ref{th:lpsanov} as a non-exponential version of the upper bound of Sanov's theorem, and the proof is given in Section \ref{se:shortfall}. At this level of generality, there cannot be a matching lower bound for open sets as in the classical case \eqref{def:classicalsanov}, as will be explained more in Section \ref{se:intro:Cramer}. Of course, Sanov's theorem applies without any moment assumptions, but the upper bound provides no information in many heavy-tailed contexts. We illustrate this with three applications below, all of which take advantage of the crucial fact that Theorem \ref{th:lpsanov} applies to arbitrary closed sets $A$, which enables a natural contraction principle (i.e., continuous mapping). The first example gives new results on the rate of convergence of empirical measures in Wasserstein distance. Second, we derive non-exponential upper bounds analogous to Cram\'er's theorem for sums of i.i.d.\ random variables with values in Banach spaces. Lastly, we derive error bounds for the usual Monte Carlo scheme in stochastic optimization, essentially providing a heavy-tailed analogue of the results of \cite{kaniovski-king-wets}.

\subsubsection{Rate of convergence of empirical measures in Wasserstein distance}
First, some terminology: A \emph{compatible metric} on $E$ is any metric on $E$ which generates the given Polish topology. For $\mu,\nu \in \P(E)$ and $q \ge 1$ define the $q$-Wasserstein distance $\W_q(\mu,\nu)$ by
\begin{align}
\W^q_q(\mu,\nu) = \inf_{\pi \in \Pi(\mu,\nu)}\int_{E \times E} d^q(x,y) \pi(dx,dy), \label{def:wasserstein}
\end{align}
where $\Pi(\mu,\nu)$ is the set of probability measures on $E \times E$ with first marginal $\mu$ and second marginal $\nu$. We will prove in Section \ref{se:pf:lpstuff} the following:

\begin{corollary}[Wasserstein convergence rate] \label{co:wasserstein-rate}
Let $d$ be any compatible metric on $E$. Let $q > r \ge 1$, and let $\mu \in \P(E)$ satisfy $\int_E d^q(x,x_0)\mu(dx) < \infty$ for some (equivalently, for any) $x_0 \in E$. Then, for each $a > 0$,
\begin{align*}
\limsup_{n\to\infty} n^{\frac{q}{r}-1}\mu^n\left(\W_r(L_n,\mu) \ge a\right) < \infty.
\end{align*}
In particular,
\begin{align*}
\limsup_{n\to\infty} n^{q-1}\mu^n\left(\W_1(L_n,\mu) \ge a\right) < \infty.
\end{align*}
\end{corollary}

In other words, $\mu^n\left(\W_r(L_n,\mu) \ge a\right) = O(n^{1-q/r})$. 
In the $r=1$ case, a comparison with a more classical setting reveals that this rate is the right one, in a sense: Suppose $X_i$ are i.i.d.\ real-valued random variables with law $\mu$, mean zero, and $\E|X_1|^q < \infty$. Then Kantorovich duality yields the a.s.\ inequality $\frac{1}{n}\sum_{i=1}^nX_i \le \W_1\left(\frac{1}{n}\sum_{i=1}^n\delta_{X_i},\mu\right)$, and Corollary \ref{co:wasserstein-rate} gives $\PP(X_1+\cdots + X_n \ge an) = O(n^{1-q})$ for each $a > 0$. It is known in this context that $\PP(X_1 + \cdots + X_n > an) = o(n^{1-q})$, and this exponent cannot be improved under the sole assumption of a finite $q^\text{th}$ moment \cite[Chapter IX, Theorems 27-28]{petrov2012sums}. 
A similar argument in the case $r > 1$ indicates that the exponent $q/r - 1$ is sharp in the first claim of Corollary \ref{co:wasserstein-rate}.

There is by now substantial literature on rates of convergence of empirical measures of i.i.d.\ sequences in Wasserstein distance, and we refer to the recent papers \cite{fournier2015rate,weed2017sharp} for an overview of the literature and applications in quantization, interacting particle systems, etc. Nonetheless, our result seems quite new in the two key aspects of applying in heavy-tailed settings and in general spaces. First, while the $n\to\infty$ convergence rate of the \emph{expected distance} $M_n^{(r)}:=\int_{E^n}\W_r^r(L_n,\mu)\,d\mu^n$ is by now well understood, the (asymptotic) rate of convergence in $n$ of the deviation probabilities given in Corollary \ref{co:wasserstein-rate} appears to be new. Case (3) in Theorem 2 of \cite{fournier2015rate} gives some non-asymptotic bounds on these probabilities which are worse than ours in the $n\to\infty$ regime; the closest counterpart among their results is a bound of $O(n^{1-q+\epsilon})$ for any $\epsilon > 0$, but it is given only for $r < q/2$ and in Euclidean spaces.

A second novelty of Corollary \ref{co:wasserstein-rate} is that it is valid in arbitrary Polish spaces, whereas most prior literature deals with Euclidean spaces. In the setting of tail probability bounds, notable exceptions include \cite{boissard2011simple} and \cite{weed2017sharp}, which show exponential decay in $n$ of the probabilities $\mu^n\left(\W_1(L_n,\mu) \ge a\right)$ but under assumptions that the measure $\mu$ has compact support or finite exponential moments. 
In the study of the expected distances $M_n^{(r)}$, it is well known that the dimension of the underlying space must absolutely come into play (where a suitable meaning of ``dimension" must be given in the metric space setting \cite{dudley1969speed,weed2017sharp}). For example, in Euclidean space $E \subset \R^d$ with $d > 2$, $M_n^{(1)}$ is known to be asymptotic to $n^{-1/d}$, at least when $E$ is compact and $\mu$ is absolutely continuous with respect to Lebesgue measure \cite{dudley1969speed}. Corollary \ref{co:wasserstein-rate} shows that this dimension-dependence disappears from the probabilistic rate of convergence. Take note that this leads to no contradiction: writing $M_n^{(1)} = \int_0^\infty\mu^n\left(\W_1(L_n,\mu) \ge a\right)da$ and applying Corollary \ref{co:wasserstein-rate} does not imply $n^{q-1}M_n^{(1)}\to 0$, as the dominated convergence theorem does not apply here.

\subsubsection{Cram\'er's upper bound} \label{se:intro:Cramer}

While Cram\'er's theorem in full generality, like Sanov's, does not require any finite moments, the upper bound is often vacuous when the underlying random variables have heavy tails. This simple observation has driven a large and growing literature on large deviation asymptotics for sums of i.i.d.\ random variables, to be reviewed shortly. This literature is full of \emph{precise asymptotics}, mostly out of reach of our abstract framework. However, from Theorem \ref{th:lpsanov} we can derive a modest alternative to Cram\'er's upper bound which is notable in its wide applicability. See Section \ref{se:pf:lpstuff} for proof of the following:

\begin{corollary}[Cram\'er upper bound] \label{co:lpcramer}
Let $q \in (1,\infty)$, and let $E$ be a separable Banach space. Let $(X_i)_{i=1}^\infty$ be i.i.d.\ $E$-valued random variables with $\E\|X_1\|^q < \infty$. Define $\Lambda : E^* \rightarrow \R \cup \{\infty\}$ by\footnote{Here, $E^*$ denotes the continuous dual of $E$.}
\[
\Lambda(x^*) = \inf\left\{m \in \R : \E\left[[\left(1+\langle x^*,X_1\rangle - m\right)^+]^q\right] \le 1\right\},
\]
and define $\Lambda^*(x) = \sup_{x^* \in E^*}\left(\langle x^*,x\rangle - \Lambda(x^*)\right)$ for $x \in E$.
Then, for every closed set $A \subset E$,
\[
\limsup_{n\rightarrow\infty}\,n^{q-1}\,\PP\left(\frac{1}{n}\sum_{i=1}^nX_i \in A\right) \le \left(\inf_{x \in A}\Lambda^*(x)\right)^{-q}.
\]
\end{corollary}

In analogy with the classical Cram\'er's theorem, the function $\Lambda$ in Corollary \ref{co:lpcramer} plays the role of the cumulant generating function.
In both Theorem \ref{th:lpsanov} and Corollary \ref{co:lpcramer}, notice that as soon as the constant on the right-hand side is finite we may conclude that the probabilities in question are $O(n^{1-q})$, consistent with some now-standard results on one-dimensional heavy tailed sums for events of the form $A=[r,\infty)$, for $r > 0$. For instance, as we mentioned in the previous subsection, if $(X_i)_{i=1}^\infty$ are i.i.d.\ real-valued random variables with mean zero and $\E|X_1|^q < \infty$, then the sharpest result possible under these assumptions is $\PP(X_1 + \cdots + X_n > nr) = o(n^{1-q})$. 
For $q > 2$, the Fuk-Nagaev inequality gives a related non-asymptotic bound; see \cite[Corollary 1.8]{nagaev1979large}, or \cite{einmahl2008characterization} for a Banach space version.

In general, we cannot expect a matching lower bound in Corollary \ref{co:lpcramer}, and thus nor can we in Theorem \ref{th:lpsanov}.
If stronger assumptions are made on $X_i$, such as regular variation, then corresponding lower bounds are known for certain sets $A$, but it remains unclear whether or not our abstract approach can recover such lower bounds. Refer to \cite{borovkov,foss2011introduction,mikosch-nagaev} for detailed overviews of such results, as well as the more recent \cite{denisov2008large,rhee2016sample} and references therein. Indeed, \emph{precise asymptotics} require detailed assumptions on the shape of the tails of $X_i$, and this is especially true in multivariate and infinite-dimensional contexts. An interesting recent line of work extends the theory of regular variation to metric spaces \cite{de2001convergence,hult2005functional,hult2006regular,lindskog2014regularly}, but again the assumptions on the underlying $\mu$ are much stronger than mere existence of a finite moment.

The only real strengths of our Corollary \ref{co:lpcramer}, compared to the deep literature on sums of heavy-tailed random variables, is its broad applicability. It requires only finite moments, applies in general (separable) Banach spaces, and allows for arbitrary closed sets $A$, the latter point being useful in that it enables contraction principle arguments.


Before turning to the next application, it is worth mentioning a few more loosely related papers. In connection with concentration of measure, the papers of Bobkov and Ding \cite{bobkov-ding,ding2014wasserstein} studied transport inequalities involving functionals like \eqref{intro:lpentropy}, resulting in characterizations of certain non-exponential tail bounds. Less closely related, Atar et al.\ in \cite{atar2015robust} exploit a variational representation for exponential integrals involving the functional \eqref{intro:lpentropy} and show how to use it to bound, for example, a large deviation probability for one model in terms of an alternative more tractable model; their work does not, however, appear to be applicable to situations with heavy tails.

\subsubsection{Stochastic optimization} \label{se:intro:optimization}

Let $\X$  be another Polish space. Consider a continuous function $h : \X \times E \rightarrow \R$ bounded from below, and define $V : \P(E) \rightarrow \R$ by 
\[
V(\nu) = \inf_{x \in \X}\int_Eh(x,w)\nu(dw).
\]
Fix $\mu \in \P(E)$ again as a reference measure. The most common and natural approach to solving the optimization problem $V(\mu)$ numerically is to construct i.i.d.\ samples $X_1,X_2,\ldots$ with law $\mu$ and study instead $V(L_n(X_1,\ldots,X_n))$, where as usual $L_n(X_1,\ldots,X_n)=\frac{1}{n}\sum_{i=1}^n\delta_{X_i}$. The two obvious questions are then:
\begin{enumerate}[(A)]
\item Does $V(L_n(X_1,\ldots,X_n))$ converge to $V(\mu)$?
\item Do the minimizers of $V(L_n(X_1,\ldots,X_n))$ converge to those of $V(\mu)$ in some sense?
\end{enumerate}
The answers to these questions are known to be affirmative in very general settings, using a form of set-convergence for question (B); see \cite{dupacova1988asymptotic,kall1987approximations,king1991epi}. Given this, we then hope to quantify the rate of convergence for both of these questions. This is done in the language of large deviations in a paper of Kaniovski et al. \cite{kaniovski-king-wets}, under a strong exponential integrability assumption derived from Cram\'er's condition. In this section we complement their results by showing that under weaker integrability assumptions we can still obtain polynomial rates of convergence:

\begin{theorem} \label{th:valueconvergence}
Suppose $\X$ is compact. Suppose the function $h$ is jointly continuous, and its sub-level sets are compact. Let $q \in (1,\infty)$ and $\mu \in \P(E)$ be such that, if
\[
\psi(w) := \left(\sup_{x \in \X}h(x,w)\right)^+,
\]
then $\int_E\psi^q\,d\mu < \infty$. Then, for each $\epsilon > 0$,
\[
\limsup_{n\rightarrow\infty}n^{q-1}\mu^n(|V(L_n)-V(\mu)| \ge \epsilon) < \infty.
\]
\end{theorem}

The proof is given in Section \ref{se:optimization}, where we also present a related result on the rate of convergence of the optimizers themselves, addressing question (B) above.

\subsection{Uniform upper bounds and martingales} \label{se:intro:uniform}

Certain classes of dependent sequences admit uniform upper bounds, which we derive from Theorem \ref{th:main-sanov} by working with
\begin{align}
\alpha(\nu) = \inf_{\mu \in M}H(\nu|\mu), \label{intro:robustentropic}
\end{align}
for a given convex weakly compact set $M \subset \P(E)$. The conjugate $\rho$, unsurprisingly, is $\rho(f) = \sup_{\mu \in M}\log\int e^f\,d\mu$, and $\rho_n$ turns out to be tractable as well:
\begin{align}
\rho_n(f) = \sup_{\mu \in M_n}\log\int_{E^n} e^f\,d\mu, \label{intro:robustentropic-rhon}
\end{align}
where $M_n$ is defined as the set of laws $\mu \in \P(E^n)$ with $\mu_{k-1,k} \in M$ for each $k=1,\ldots,n$, $\mu$-almost surely; in other words, $M_n$ is the set of laws of $E^n$-valued random variables $(X_1,\ldots,X_n)$, when the law of $X_1$ belongs to $M$ and so does the conditional law of $X_k$ given $(X_1,\ldots,X_{k-1})$, almost surely, for each $k=2,\ldots,n$. Theorem \ref{th:main-sanov} becomes
\[
\lim_{n\rightarrow\infty}\frac{1}{n}\log\sup_{\mu \in M_n}\int_{E^n}e^{nF\circ L_n}\,d\mu = \sup_{\mu \in M, \ \nu \in \P(E)}(F(\nu) - H(\nu|\mu)), \text{ for } F \in C_b(\P(E)).
\]
From this we derive a uniform large deviation upper bound, for closed sets $A \subset \P(E)$:
\begin{align}
\limsup_{n\rightarrow\infty}\frac{1}{n}\log\sup_{\mu \in M_n}\mu(L_n \in A) \le -\inf_{\mu \in M, \nu \in A}H(\nu | \mu). \label{intro:uniformLDP}
\end{align}
With a prudent choice of $M$, this specializes to an asymptotic relative of the Azuma-Hoeffding inequality. The novel feature here is that we can work with arbitrary closed sets and in multiple dimensions:

\begin{theorem} \label{th:azuma}
Let $\varphi : \R^d \rightarrow \R$, and define $\mathcal{S}_{d,\varphi}$ to be the set of $\R^d$-valued martingales $(S_k)_{k=0}^n$, defined on a common but arbitrary probability space, satisfying $S_0=0$ and
\[
\E\left[\left.\exp\left(\langle y, S_k-S_{k-1}\rangle\right)\right|S_0,\ldots,S_{k-1}\right] \le e^{\varphi(y)}, \ a.s., \text{ for } k=1,\ldots,n, \ y \in \R^d.
\]
Then, for closed sets $A \subset \R^d$, we have
\[
\limsup_{n\rightarrow\infty}\sup_{(S_k)_{k=0}^n \in \mathcal{S}_{d,\varphi}}\frac{1}{n}\log\PP\left(S_n/n \in A\right) \le -\inf_{x \in A}\varphi^*(x),
\]
where $\varphi^*(x) = \sup_{y \in \R^d}(\langle x,y\rangle - \varphi(y))$.
\end{theorem}

By taking $(S_k)_{k=0}^n$ to be a random walk (that is, the increments are i.i.d.)\ such that $\varphi(y) = \log\E[e^{\langle y,S_1-S_0\rangle}] < \infty$,  it is readily checked that the bound of Theorem \ref{th:azuma} coincides with the upper bound from Cram\'er's theorem and is thus sharp.
F\"ollmer and Knispel \cite{follmer-knispel} found some results which loosely resemble \eqref{intro:uniformLDP} (see Corollary 5.3 therein), based on an analysis of the same risk measure $\rho$. See also \cite{hu2010cramer,fuqing2012relative} for somewhat related results on large deviations for capacities.

\subsection{Laws of large numbers} \label{se:intro:lln}
Some specializations of Theorem \ref{th:main-sanov} appear to have nothing to do with large deviations. For example, suppose $M \subset \P(E)$ is convex and compact, and let
\[
\alpha(\nu) = \begin{cases}
0 &\text{if } \nu \in M \\
\infty &\text{otherwise}.
\end{cases}.
\]
It can be shown that $\rho_n(f) = \sup_{\mu \in M_n}\int_{E^n}f\,d\mu$, where $M_n$ is defined as in Section \ref{se:intro:uniform}, for instance by a direct computation using \eqref{def:intro:rhon-recursive}. Theorem \ref{th:main-sanov} then becomes
\begin{align}
\lim_{n\rightarrow\infty}\sup_{\mu \in M_n}\int_{E^n}F \circ L_n\,d\mu = \sup_{\mu \in M}F(\mu), \text{ for each } F \in C_b(\P(E)). \label{def:robustLLN}
\end{align}
When $M =\{\mu\}$ is a singleton, so is $M_n = \{\mu^n\}$, and this simply expresses the weak convergence $\mu^n \circ L_n^{-1} \rightarrow \delta_\mu$. The general case can be interpreted as a robust law of large numbers, where ``robust'' refers to perturbations of the joint law of an i.i.d.\ sequence. More precisely, noting that $\sup_{\mu \in M}F(\mu) = \sup_{Q \in \P(M)}\int F\,dQ$,  one can derive from \eqref{def:robustLLN} certain forms of set-convergence (e.g., Painlev\'e-Kuratowski) of the sequence $\{\mu \circ L_n^{-1} : \mu \in M_n\}$ toward $\P(M) := \{Q \in \P(\P(E)) : Q(M)=1\}$, though we refrain from lengthening the paper with further details.
In another direction, \eqref{def:robustLLN} is closely related to laws of large numbers under nonlinear expectations \cite{peng2007law}.

\subsection{Optimal transport costs} \label{se:intro:optimaltransport}

Another interesting consequence of Theorem \ref{th:main-sanov} comes from choosing $\alpha$ as an optimal transport cost.
Fix  $\mu \in \P(E)$ and a lower semicontinuous function $c : E^2 \rightarrow [0,\infty]$, and define
\[
\alpha(\nu) = \inf_{\pi \in \Pi(\mu,\nu)}\int c\,d\pi,
\]
where $\Pi(\mu,\nu)$ was defined right after \eqref{def:wasserstein}. Under a modest additional assumption on $c$ (stated shortly in Corollary \ref{co:transport1}, proven later in Lemma \ref{le:tightnessfunction}), $\alpha$ satisfies our standing assumptions.

The dual $\rho$ can be identified using Kantorovich duality, and $\rho_n$ turns out to be the value of a stochastic optimal control problem. To illustrate this, it is convenient to work with probabilistic notation: Suppose $(X_i)_{i=1}^\infty$ is a sequence of i.i.d.\ $E$-valued random variables with common law $\mu$, defined on some fixed probability space. For each $n$, let $\Y_n$ denote the set of $E^n$-valued random variables $(Y_1,\ldots,Y_n)$ where $Y_k$ is $(X_1,\ldots,X_k)$-measurable for each $k=1,\ldots,n$. We think of elements of $\Y_n$ as adapted control processes. For each $n \ge 1$ and each $f \in B(E^n)$, we show in Proposition \ref{pr:rhon-optimaltransport} that
\begin{align}
\rho_n(f) = \sup_{(Y_1,\ldots,Y_n) \in \Y_n}\E\left[f(Y_1,\ldots,Y_n) - \sum_{i=1}^nc(X_i,Y_i)\right]. \label{def:optimaltransport-expression}
\end{align}
The expression \eqref{def:optimaltransport-expression} yields the following corollary of Theorem \ref{th:main-sanov}:

\begin{corollary} \label{co:transport1}
Suppose that for each compact set $K \subset E$, the function $h_K(y) := \inf_{x \in K}c(x,y)$ has pre-compact sub-level sets.\footnote{That is, the closure of $\{y \in E : h_K(y) \le m\}$ is compact for each $m \ge 0$. This assumption holds, for example, if $E$ is a subset of Euclidean space and there exists $y_0 \in E$ such that $c(x,y) \rightarrow \infty$ as $d(y,y_0) \rightarrow \infty$, uniformly for $x$ in compacts.} For each $F \in C_b(\P(E))$, we have
\begin{align}
\lim_{n\rightarrow\infty}\sup_{(Y_k)_{k=1}^n \in \Y_n}&\E\left[F(L_n(Y_1,\ldots,Y_n)) - \frac{1}{n}\sum_{i=1}^nc(X_i,Y_i)\right] = \sup_{\nu \in \P(E)}\left(F(\nu) - \alpha(\nu)\right) \nonumber \\
	&= \sup_{\pi \in \Pi(\mu)}\left(F(\pi(E \times \cdot)) - \int_{E \times E}c\,d\pi\right), \label{def:optimaltransport-mainlimit}
\end{align}
where $\Pi(\mu) = \cup_{\nu \in \P(E)}\Pi(\mu,\nu)$.
\end{corollary}

This can be seen as a long-time limit of the optimal value of the control problems. However, the renormalization in $n$ is a bit peculiar in that it enters inside of the terminal cost $F$, and there does not seem to be a direct connection with ergodic control. A direct proof of \eqref{def:optimaltransport-mainlimit} is possible but seems to be no simpler and potentially narrower in scope.

While the pre-limit expression in \eqref{def:optimaltransport-mainlimit} may look peculiar, we include this example in part because it is remarkably tractable and  in part because the limiting object is quite ubiquitous, encompassing a wide variety of variational problems involving optimal transport costs. Two notable recent examples can be found in the study of Cournot-Nash equilibria in large-population games \cite{blanchet2015optimal} and in the theory of Wasserstein barycenters \cite{agueh2011barycenters}.

\subsection{Alternative choices of $\alpha$} \label{se:intro:alternatives}
There are many other natural choices of $\alpha$ for which the implications of Theorem \ref{th:main-sanov} remain unclear. For example, consider the $\varphi$-divergence
\[
\alpha(\nu) = \int_E\varphi(d\nu/d\mu)\,d\mu, \text{ for } \nu \ll \mu, \quad\quad \alpha(\nu)=\infty \text{ otherwise},
\]
where $\mu \in \P(E)$ and $\varphi : \R_+ \rightarrow \R$ is convex and satisfies $\varphi(x)/x \rightarrow \infty$ as $x \rightarrow \infty$. This $\alpha$ has weakly compact sub-level sets, according to \cite[Lemma 6.2.16]{dembozeitouni}, and it is clearly convex. The dual, known in the risk literature as the \emph{optimized certainty equivalent,} was computed by Ben-Tal and Teboulle \cite{bental-teboulle-1986,bental-teboulle-2007} to be
\[
\rho(f) = \inf_{m \in \R}\left(\int_E\varphi^*\left(f(x)-m\right)\mu(dx) + m\right),
\]
where $\varphi^*(x) = \sup_{y \in \R}(xy - \varphi(y))$ is the convex conjugate.
We did not find any good expressions or estimates for $\rho_n$ or $\alpha_n$, so the interpretation of the main Theorem \ref{th:main-sanov} eludes us in this case.

A related choice is the \emph{shortfall risk measure} introduced by F\"ollmer and Schied \cite{follmer-schied-convex}:
\begin{align}
\rho(f) = \inf\left\{m \in \R : \int_E\ell(f(x)-m)\mu(dx) \le 1\right\}. \label{intro:shortfall}
\end{align}
This choice of $\rho$ and the corresponding (tractable!) $\alpha$ are discussed briefly in Section \ref{se:shortfall}.
The choice of $\ell(x) = [(1+x)^+]^q$ corresponds to \eqref{intro:lpentropy}, and we make extensive use of this in Section \ref{se:nonexpLDP}, as was discussed in Section \ref{se:intro:nonexpLDP}.
The choice of $\ell(x) = e^x$ recovers the classical case $\rho(f) = \log\int_Ee^f\,d\mu$.
Aside from these two examples, for general $\ell$, we found no useful expressions or estimates for $\rho_n$ or $\alpha_n$.
In connection with tails of random variables, shortfall risk measures have an intuitive appeal stemming from the following simple analogue of Chernoff's bound, observed in \cite[Proposition 3.3]{lacker-liquidity}: If $\gamma(\lambda) = \rho(\lambda f)$ for all $\lambda \ge 0$, where $f$ is some given measurable function, then $\mu(f > t) \le 1/\ell(\gamma^*(t))$ for all $t \ge 0$, where $\gamma^*(t) = \sup_{\lambda \ge 0}(\lambda t - \gamma(\lambda))$.

It is worth pointing out the natural but ultimately fruitless idea of working with $\rho(f) = \varphi^{-1}(\int_E\varphi(f)\,d\mu)$, where $\varphi$ is increasing. Such functionals were studied first it seems by Hardy, Littlewood, and P\'olya \cite[Chapter 3]{hardy-littlewood-polya}, providing necessary and sufficient conditions for $\rho$ to be convex  (rediscovered in \cite{bental-teboulle-1986}). Using the formula \eqref{def:intro:rhon-recursive} to compute $\rho_n$, this choice would lead to the exceptionally pleasant formula $\rho_n(f) = \varphi^{-1}(\int_{E^n}\varphi(f)\,d\mu^n)$, which we observed already in the classical case $\varphi(x)=e^x$. Unfortunately, however, such a $\rho$ cannot come from a functional $\alpha$ on $\P(E)$, in the sense that \eqref{intro:duality} cannot hold unless $\varphi$ is affine or exponential.
The problem, as is known in the risk measure literature, is that the additivity property $\rho(f+c)=\rho(f)+c$ for all $c \in \R$ and $f \in B(E)$ fails unless $\varphi$ is affine or exponential (c.f. \cite[Proposition 2.46]{follmer-schied-book}).

\subsection{Interpreting Theorem \ref{th:main-sanov} in terms of risk measures} \label{se:intro:riskmeasures}
It is straightforward to rewrite Theorem \ref{th:main-sanov} in a language more in line with the literature on convex risk measures, for which we again defer to \cite{follmer-schied-book} for background.
Let $(\Omega,\F)$ be a measurable space, and suppose $\varphi$ is a convex risk measure on the set $B(\Omega,\F)$ of bounded measurable functions. That is, $\varphi : B(\Omega,\F) \rightarrow \R$ is convex, $\varphi(f + c) = \varphi(f)+c$ for all $f \in B(\Omega,\F)$ and $c \in \R$, and $\varphi(f) \ge \varphi(g)$ whenever $f \ge g$ pointwise. Suppose we are given a sequence of $E$-valued random variables $(X_i)_{i=1}^\infty$, i.e., measurable maps $X_i : \Omega \rightarrow E$. Assume $X_i$ have the following independence property, identical to Peng's notion of independence under nonlinear expectations \cite{peng2010nonlinear}: for $n \ge 1$ and $f \in B(E^n)$
\begin{align}
\varphi(f(X_1,\ldots,X_n)) = \varphi\left[\varphi(f(X_1,\ldots,X_{n-1},x))|_{x = X_n}\right]. \label{def:peng-independence}
\end{align}
In particular, $\varphi(f(X_i))=\varphi(f(X_1))$ for all $i$.
Define $\alpha : \P(E) \rightarrow (-\infty,\infty]$ by
\[
\alpha(\nu) = \sup_{f \in B(E)}\left(\int_Ef\,d\nu - \varphi(f(X_1))\right).
\]
Additional assumptions on $\varphi$ (see, e.g., Theorem \ref{th:rho-to-alpha} below) can ensure that $\alpha$ has weakly compact sub-level sets, so that Theorem \ref{th:main-sanov} applies. Then, for $F \in C_b(\P(E))$, 
\begin{align}
\lim_{n\rightarrow\infty}\frac{1}{n}\varphi\left(nF(L_n(X_1,\ldots,X_n))\right) = \sup_{\nu \in \P(E)}(F(\nu)-\alpha(\nu)), \label{intro:def:riskmeasure-interpretation}
\end{align}
Indeed, in our previous notation, $\rho_n(f)=\varphi(f(X_1,\ldots,X_n))$  for $f \in B(E^n)$.

In the risk measure literature, one thinks of $\varphi(f)$ as the risk associated to an uncertain financial loss $f \in B(\Omega,\F)$. With this in mind, and with $Z_n=F(L_n(X_1,\ldots,X_n))$, the quantity $\varphi(nZ_n)$ appearing in \eqref{intro:def:riskmeasure-interpretation} is the risk-per-unit of an investment in $n$ units of $Z_n$. One might interpret $Z_n$ as capturing the \emph{composition} of the investment, while the multiplicative factor $n$ represents the \emph{size} of the investment. As $n$ increases, say to $n+1$, the investment is ``rebalanced'' in the sense that one additional independent component, $X_{n+1}$, is incorporated and the size of the total investment is increased by one unit. The limit in \eqref{intro:def:riskmeasure-interpretation} is then an asymptotic evaluation of the risk-per-unit of this rebalancing scheme.

\subsection{Extensions}

Broadly speaking, the book of Dupuis and Ellis \cite{dupuis-ellis} and numerous subsequent works illustrate how the classical convex duality between relative entropy and cumulant-generating functions can serve as a foundation from which to derive an impressive range of large deviation principles.
Similarly, each alternative dual pair  $(\alpha,\rho)$ should provide an alternative foundation for a potentially equally wide range of limit theorems. From this perspective, our work raises more questions than it answers by restricting attention to analogs of the two large deviation principles of Sanov and Cram\'er.
It is possible, for instance, that an analogue of Mogulskii's theorem (see \cite{mogul1977large} or \cite[Section 3]{dupuis-ellis}) holds in our context, though one must not expect any such analogue to look too much like a heavy-tailed large deviation principle, in light of the negative result of \cite[Section 4.4]{rhee2016sample}.\footnote{Thanks to an anonymous referee for this reference.} These speculations are pursued no further but are meant to convey the versatility of our framework.
In fact, extensions and applications of our framework have appeared since the first version of this paper. First, \cite{eckstein2017extended} extended the ideas beyond the i.i.d.\ setting, to the study of occupation measures of Markov chains. More recently, \cite{backhoff2018non} applied Theorem \ref{th:main-sanov} to obtain new limit theorems for Brownian motion, with connections to Schilder's theorem, vanishing noise limits of BSDEs and PDEs, and Schr\"{o}dinger problems.

\subsection{Outline of the paper}

The remainder of the paper is organized as follows. Section \ref{se:riskmeasures} begins by clarifying the $(\alpha,\rho)$ duality, explaining some useful properties of $\rho$ and $\rho_n$ and extending their definitions to unbounded functions before giving. In Section \ref{se:sanovproof} we give the proof of an extension of Theorem \ref{th:main-sanov}, which contains Theorem \ref{th:main-sanov} as a special case but is extended to stronger topologies and unbounded functions $F$. See also Section \ref{se:contraction} for abstract analogs of the contraction principle and Cram\'er's theorem. Section \ref{se:psi-compact-levelsets} elaborates on the additional topological assumptions needed for the extension of the main theorem. Then, Section \ref{se:nonexpLDP} focuses on the particular choice of $\alpha$ in \eqref{intro:lpentropy}, providing proofs of the claims of Section \ref{se:intro:nonexpLDP}. 
Sections \ref{se:uniform} and \ref{se:optimaltransport} respectively elaborate on the examples of \ref{se:intro:uniform} and \ref{se:intro:optimaltransport}. Appendix \ref{se:rhon} proves a different representations of $\rho_n$, namely those of \eqref{def:intro:rhon-recursive}.
Finally, two minor technical results are relegated to Appendix \ref{se:technical-lemma}.

\section{Convex duality and an extension of Theorem \ref{th:main-sanov}} \label{se:riskmeasures}

We begin by outlining the key features of the $(\alpha,\rho)$ duality, as a first step toward stating and proving an extension of the main theorem as well as an abstract analogue of Cram\'er's theorem. The first two theorems below are borrowed from the literature on convex risk measures, for which an excellent reference is the book of F\"ollmer and Schied \cite{follmer-schied-book}.
While we will make use of some of the properties listed in Theorem \ref{th:alpha-to-rho}, the goal of the first two theorems is more to illustrate how one can make $\rho$ the starting point rather than $\alpha$. In particular, Theorem \ref{th:rho-to-alpha} will not be needed in the sequel. For proofs of Theorems \ref{th:alpha-to-rho} and \ref{th:rho-to-alpha}, refer to Bartl \cite[Theorem 2.6]{bartl2016pointwise}.

\begin{theorem} \label{th:alpha-to-rho}
Suppose $\alpha : \P(E) \rightarrow (-\infty,\infty]$ is convex and has weakly compact sub-level sets. Define $\rho : B(E) \rightarrow \R$ as in \eqref{intro:duality}. Then the following hold:
\begin{enumerate}
\item[(R1)] If $f \ge g$ pointwise then $\rho(f) \ge \rho(g)$.
\item[(R2)] If $f \in B(E)$ and $c \in \R$, then $\rho(f+c)=\rho(f)+c$.
\item[(R3)] If $f,f_n \in B(E)$ with $f_n \uparrow f$ pointwise, then $\rho(f_n) \uparrow \rho(f)$.
\item[(R4)] If $f_n \in C_b(E)$ and $f \in B(E)$ with $f_n \downarrow f$ pointwise, then $\rho(f_n) \downarrow \rho(f)$.
\end{enumerate}
Moreover, for $\nu \in \P(E)$ we have
\begin{align}
\alpha(\nu) = \sup_{f \in C_b(E)}\left(\int_Ef\,d\nu - \rho(f)\right). \label{def:rho-to-alpha}
\end{align}
\end{theorem}

\begin{theorem} \label{th:rho-to-alpha} 
Suppose $\rho : B(E) \rightarrow \R$ is convex and satisfies properties (R1--4) of Theorem \ref{th:alpha-to-rho}.
Define $\alpha : \P(E) \rightarrow (-\infty,\infty]$ by \eqref{def:rho-to-alpha}.
Then $\alpha$ is convex and has weakly compact sub-level sets. Moreover, the identity \eqref{intro:duality} holds.
\end{theorem}

For the rest of the paper, unless stated otherwise, we work at all times with the standing assumptions on $\alpha$ described in the introduction:

\begin{assumption}
The function $\alpha : \P(E) \rightarrow (-\infty,\infty]$ is convex, has weakly compact sub-level sets, and is not identically equal to $\infty$. Lastly, $\rho$ is defined as in \eqref{intro:duality}.
\end{assumption}

We next extend the domain of $\rho$ and $\rho_n$ to unbounded functions. Let $\overline{\R} = \R \cup \{-\infty,\infty\}$. We adopt the convention that $\infty - \infty := -\infty$, although this will have few consequences aside from streamlined definitions. In particular, if $\nu \in \P(E^n)$ and a measurable function $f : E^n \rightarrow \overline{\R}$ and satisfies $\int f^-\,d\nu = \int f^+\,d\nu = \infty$, we define $\int f\,d\nu = -\infty$.

\begin{definition} \label{def:rho-unbounded}
For $n \ge 1$ and measurable $f : E^n \rightarrow \overline{\R}$, define
\[
\rho_n(f) = \sup_{\nu \in \P(E^n)}\left(\int_{E^n}f\,d\nu - \alpha_n(\nu)\right).
\]
As usual, abbreviate $\rho \equiv \rho_1$. It is worth emphasizing that while $\rho(f)$ is finite for bounded $f$, it can be either $+\infty$ or $-\infty$ when $f$ is unbounded.
\end{definition}

%
%
%

\subsection{Stronger topologies on $\P(E)$} \label{se:strongertopologies}

As a last preparation, we discuss a well known class of topologies on subsets of $\P(E)$ with which we will work frequently.
Given a continuous function $\psi : E \rightarrow \R_+ := [0,\infty)$, define
\[
\P_\psi(E) = \left\{\mu \in \P(E) : \int_E\psi\,d\mu < \infty\right\}.
\]
Endow $\P_\psi(E)$ with the (Polish) topology generated by the maps $\nu \mapsto \int_Ef\,d\nu$, where $f : E \rightarrow \R$ is continuous and $|f| \le 1+\psi$; we call this the \emph{$\psi$-weak topology}.
A useful fact about this topology is that a set $M \subset \P_\psi(E)$ is pre-compact if and only if for every $\epsilon > 0$ there exists a compact set $K \subset E$ such that
\[
\sup_{\mu \in M}\int_{K^c}\psi\,d\mu \le \epsilon.
\]
This is easily proven directly using Prokhorov's theorem, or refer to \cite[Corollary A.47]{follmer-schied-book}. It is worth noting that if $d$ is a compatible metric on $E$ and $\psi(x)=d^p(x,x_0)$ for some fixed $x_0 \in E$ and $p \ge 1$, then the $\psi$-weak topology is nothing but the $p$-Wasserstein topology associated with the metric $d$ \cite[Theorem 7.12]{villani-book}.

\subsection{An extension of Theorem \ref{th:main-sanov}} \label{se:sanovproof}

In this section we state and prove a useful generalization of Theorem \ref{th:main-sanov} for stronger topologies and unbounded functions, taking advantage of the preparations of the previous sections.
At all times in this section, the standing assumptions on $(\alpha,\rho)$ (stated early in Section \ref{se:riskmeasures}) are in force.
Part of Theorem \ref{th:main-sanov-extended} below requires the assumption that the sub-level sets of $\alpha$ are pre-compact in $\P_\psi(E)$, and this rather opaque assumption will be explored in more detail in Section \ref{se:strongcramer}.

\begin{theorem} \label{th:main-sanov-extended}
Let $\psi : E \rightarrow \R_+$ be continuous. If $F : \P_\psi(E) \rightarrow \R \cup \{\infty\}$ is lower semicontinuous (with respect to the $\psi$-weak topology) and bounded from below, then 
\[
\liminf_{n\rightarrow\infty}\frac{1}{n}\rho_n(nF \circ L_n) \ge \sup_{\nu \in \P_\psi(E)}(F(\nu) - \alpha(\nu)).
\]
Suppose also that the sub-level sets of $\alpha$ are pre-compact subsets of $\P_\psi(E)$. 
If $F : \P_\psi(E) \rightarrow \R \cup \{-\infty\}$ is upper semicontinuous and bounded from above, then 
\[
\limsup_{n\rightarrow\infty}\frac{1}{n}\rho_n(nF \circ L_n) \le \sup_{\nu \in \P_\psi(E)}(F(\nu) - \alpha(\nu)).
\]
\end{theorem}
\begin{proof} {\ }

\textbf{Lower bound:} Let us prove first the lower bound.
It is immediate from the definition that $n^{-1}\alpha_n(\nu^n) = \alpha(\nu)$ for each $\nu \in \P(E)$, recalling that $\nu^n$ denotes the $n$-fold product measure.
Thus
\begin{align}
\frac{1}{n}\rho_n(nF(L_n)) &= \sup_{\nu \in \P(E^n)}\left\{\int_{E^n} F \circ L_n \,d\nu - \frac{1}{n}\alpha_n(\nu)\right\} \label{pf:sanov1} \\
	&\ge \sup_{\nu \in \P(E)}\left\{\int_{E^n} F \circ L_n \,d\nu^n - \frac{1}{n}\alpha_n(\nu^n)\right\} \nonumber \\
	&= \sup_{\nu \in \P(E)}\left\{\int_{E^n} F \circ L_n \,d\nu^n - \alpha(\nu)\right\}. \nonumber
\end{align}
For $\nu \in \P(E)$, the law of large numbers (see \cite[Theorem 11.4.1]{dudley2018real}) implies $\nu^n \circ L_n^{-1} \rightarrow \delta_\nu$ weakly, i.e., in $\P(\P(E))$. For $\nu \in \P_{\psi}(E)$, the convergence takes place in $\P(\P_{\psi}(E))$. Lower semicontinuity of $F$ on $\P_\psi(E)$ then implies (e.g., by \cite[Theorem A.3.12]{dupuis-ellis}), for each $\nu \in \P_\psi(E)$,
\begin{align*}
\liminf_{n\rightarrow\infty}\frac{1}{n}\rho_n(nF(L_n)) &\ge \liminf_{n\rightarrow\infty}\int_{E^n} F \circ L_n \,d\nu^n - \alpha(\nu) \\
	&\ge F(\nu) - \alpha(\nu).
\end{align*}
Take the supremum over $\nu$ to complete the proof of the lower bound.

\textbf{Upper bound, $F$ bounded:} 
The upper bound is more involved. First we prove it in four steps under the assumption that $F$ is bounded.

\textbf{Step 1:} 
First we simplify the expression somewhat. For each $\nu \in \P(E^n)$ the definition of $\alpha_n$ and convexity of $\alpha$ imply
\begin{align*}
\frac{1}{n}\alpha_n(\nu) &= \frac{1}{n}\sum_{k=1}^n\int_{E^n}\alpha(\nu_{k-1,k}(x_1,\ldots,x_{k-1}) )\nu(dx_1,\ldots,dx_n) \\
	&\ge \int_{E^n}\alpha\left(\frac{1}{n}\sum_{k=1}^n\nu_{k-1,k}(x_1,\ldots,x_{k-1})\right)\nu(dx_1,\ldots,dx_n).
\end{align*}
Combine this with \eqref{pf:sanov1} to get
\begin{align}
\frac{1}{n}\rho_n(nF(L_n)) &\le \sup_{\nu \in \P(E^n)}\int_{E^n}\left[F(L_n) - \alpha\left(\frac{1}{n}\sum_{k=1}^n\nu_{k-1,k}\right)\right]\,d\nu. \label{pf:sanov1.1}
\end{align}
Now choose arbitrarily some $\mu_f$ such that $\alpha(\mu_f) < \infty$.
The choice $\nu = \mu_f^n$ and boundedness of $F$ show that the supremum in  \eqref{pf:sanov1.1} is bounded below by $-\|F\|_{\infty} - \alpha(\mu_f)$, where $\|F\|_\infty := \sup_{\nu \in \P_\psi(E)}|F(\nu)|$. For each $n$, choose $\nu^{(n)} \in \P(E^n)$ attaining the supremum in \eqref{pf:sanov1.1} to within $1/n$. Then
\begin{align}
\int_{E^n}\alpha\left(\frac{1}{n}\sum_{k=1}^n\nu^{(n)}_{k-1,k}\right)\,d\nu^{(n)} \le 2\|F\|_{\infty} + \alpha(\mu_f) + \frac{1}{n}. \label{pf:sanov1.1.01}
\end{align}
It is convenient to switch now to a probabilistic notation: On some sufficiently rich probability space, find an $E^n$-valued random variable $(Y^n_1,\ldots,Y^n_n)$ with law $\nu^{(n)}$. Define the random measures
\[
S_n := \frac{1}{n}\sum_{k=1}^n\nu^{(n)}_{k-1,k}(Y^n_1,\ldots,Y^n_{k-1}), \quad\quad \widetilde{S}_n := \frac{1}{n}\sum_{k=1}^n\delta_{Y^n_k}.
\]
Use \eqref{pf:sanov1.1} and unwrap the definitions to find
\begin{align}
\frac{1}{n}\rho_n(nF(L_n)) &\le \E[F(\widetilde{S}_n) - \alpha(S_n)] + 1/n. \label{pf:sanov1.1.1}
\end{align}
Moreover, \eqref{pf:sanov1.1.01} implies
\begin{align}
\sup_n\E[\alpha(S_n)] \le 2\|F\|_\infty + \alpha(\mu_f) + 1 < \infty. \label{pf:sanov1.2.02}
\end{align}

\textbf{Step 2:} 
We next show that the sequence $(S_n,\widetilde{S}_n)$ is tight, viewed as $\P_\psi(E)\times\P_\psi(E)$-valued random variables. Here we use the assumption that the sub-level sets of $\alpha$ are $\psi$-weakly compact subsets of $\P_\psi(E)$. It then follows from \eqref{pf:sanov1.2.02} that $(S_n)$ is tight (see, e.g., \cite[Theorem A.3.17]{dupuis-ellis}). 

To see that the pair $(S_n,\widetilde{S}_n)$ is tight, it remains to check that $(\widetilde{S}_n)_n$ is tight. To this end, we first notice that $S_n$ and $\widetilde{S}_n$ have the same mean measure for each $n$, in the sense that for every $f \in B(E)$ we have
\begin{align}
\E\left[\int_Ef\,dS_n\right] &= \E\left[\frac{1}{n}\sum_{k=1}^n\E\left[\left. f(Y^n_k)\right|Y^n_1,\ldots,Y^n_{k-1}\right]\right] =  \E\left[\frac{1}{n}\sum_{k=1}^nf(Y^n_k)\right] = \E\left[\int_Ef\,d\widetilde{S}_n\right]. \label{pf:meanmeasures}
\end{align}
To prove $(\widetilde{S}_n)$ is tight, it suffices (by Prokhorov's theorem) to show that for all $\epsilon > 0$ there exists a $\psi$-weakly compact set $K \subset \P_\psi(E)$ such that $P(\widetilde{S}_n \notin K) \le \epsilon$. We will look for $K$ of the form $K = \cap_{k=1}^\infty\{\nu : \int_{C_k^c}\psi\,d\nu \le 1/k\}$, where $(C_k)_{k=1}^\infty$ a sequence of compact subsets of $E$ to be specified later; indeed, sets $K$ of this form are pre-compact in $\P_\psi(E)$ according to a form of Prokhorov's theorem discussed in Section \ref{se:strongertopologies} (see also \cite[Corollary A.47]{follmer-schied-book}). For such a set $K$, use Markov's inequality and \eqref{pf:meanmeasures} to compute
\begin{align}
P\left(\widetilde{S}_n \notin K \right) &\le \sum_{k=1}^\infty P\left(\int_{C_k^c}\psi\,d\widetilde{S}_n> 1/k\right) \le \sum_{k=1}^\infty k\,\E\int_{C_k^c}\psi\,d\widetilde{S}_n = \sum_{k=1}^\infty k\,\E\int_{C_k^c}\psi\,dS_n. \label{pf:cramer1}
\end{align}
By a form of Jensen's inequality (see Lemma \ref{le:jensen}),
\[
\sup_n\alpha(\E S_n)  \le \sup_n\E[\alpha(S_n)] < \infty,
\]
where $\E S_n$ is the probability measure on $E$ defined by $(\E S_n)(A) = \E[S_n(A)]$.
Hence, the sequence $(\E S_n)$ is pre-compact in $\P_\psi(E)$, thanks to the assumption that sub-level sets of $\alpha$ are pre-compact subsets of $\P_\psi(E)$.
It follows that for every $\epsilon > 0$ there exists a compact set $C \subset E$ such that $\sup_n\E\int_{C^c}\psi\,dS_n \le \epsilon$. With this in mind, we may choose $C_k$ to make \eqref{pf:cramer1} arbitrarily small, uniformly in $n$. This shows that $(\widetilde{S}_n)$ is tight, completing Step 2.

\textbf{Step 3:}
We next show that every limit in distribution of $(S_n,\widetilde{S}_n)$ is concentrated on the diagonal $\{(\nu,\nu) : \nu \in \P_\psi(E)\}$.
By definition of $\nu^{(n)}_{k-1,k}$, we have
\[
\E\left[\left. f(Y^n_k) - \int_E f\,d\nu^{(n)}_{k-1,k}(Y^n_1,\ldots,Y^n_{k-1})\right| Y^n_1,\ldots,Y^n_{k-1}\right] = 0, \text{ for } k=1,\ldots,n
\]
for every  $f \in B(E)$. That is, the terms inside the expectation form a martingale difference sequence. Thus, for $f \in B(E)$, we have
\begin{align}
\E\left[\left(\int_E f\,dS_n - \int_E f\, d\widetilde{S}_n\right)^2 \right]  &= \E\left[\left(\frac{1}{n}\sum_{k=1}^n\left( f(Y^n_k) - \int_E f\,d\nu^{(n)}_{k-1,k}(Y^n_1,\ldots,Y^n_{k-1})\right)\right)^2\right] \nonumber \\
	&= \frac{1}{n^2}\sum_{k=1}^n\E\left[\left( f(Y^n_k) - \int_E f\,d\nu^{(n)}_{k-1,k}(Y^n_1,\ldots,Y^n_{k-1})\right)^2\right] \nonumber  \\
	&\le 2\|f\|_{\infty}^2/n, \label{pf:sanov1.2}
\end{align}
where $\|f\|_\infty := \sup_{x \in E}|f(x)|$.
It is straightforward to check that \eqref{pf:sanov1.2} implies that every weak limit of $(S_n,\widetilde{S}_n)$ is concentrated on (i.e., almost surely belongs to) the diagonal $\{(\nu,\nu) : \nu \in \P(E)\}$ (c.f. \cite[Lemma 2.5.1(b)]{dupuis-ellis}). Indeed, if $(S,\widetilde{S})$ is some $\P_\psi(E) \times \P_\psi(E)$-valued random variable such that $(S_{n_k},\widetilde{S}_{n_k})$ converges in law to $(S,\widetilde{S})$, then \eqref{pf:sanov1.2} implies
\begin{align*}
\E\left[\left(\int_E f\,dS - \int_E f\,d\widetilde{S}\right)^2\right] = 0,
\end{align*}
for each $f \in C_b(E)$, by continuity of the map $\P_\psi(E) \times \P_\psi(E) \ni (\nu,\widetilde{\nu}) \mapsto \left(\int_E f\,d\nu - \int_E f\,d\widetilde{\nu}\right)^2$. Hence, $\int_E f\,dS = \int_E f\,d\widetilde{S}$ a.s.\ for each $f \in C_b(E)$, and arguing with a countable separating family from $C_b(E)$ (see, e.g., \cite[Theorem 6.6]{parthasarathy2005probability}) allows us to deduce that $S=\widetilde{S}$ a.s.

\textbf{Step 4:} We can now complete the proof of the upper bound.
With Step 3 in mind, fix a subsequence and a $\P_\psi(E)$-valued random variable $\eta$ such that $(S_n,\widetilde{S}_n) \rightarrow (\eta,\eta)$ in distribution (where we relabeled the subsequence). Recall that $\alpha$ is bounded from below and $\psi$-weakly lower semicontinuous, whereas $F$ is upper semicontinuous and bounded. Returning to \eqref{pf:sanov1.1.1}, we conclude now that
\begin{align*}
\limsup_{n\rightarrow\infty}\frac{1}{n}\rho_n(nF(L_n)) &\le \limsup_{n\rightarrow\infty}\E\left[F(\widetilde{S}_n) - \alpha(S_n)\right] \\
	&\le \E[F(\eta) - \alpha(\eta)] \\
	&\le \sup_{\nu \in \P_\psi(E)}\left\{F(\nu) - \alpha(\nu)\right\}.
\end{align*}
Of course, we abused notation by relabeling the subsequences, but we have argued that for every subsequence there exists a further subsequence for which this bound holds, which proves the upper bound for $F$ bounded.

\textbf{Upper bound, unbounded $F$:}
With the proof complete for bounded $F$, we now remove the boundedness assumption using a natural truncation procedure. Let $F : \P(E) \rightarrow E \cup \{-\infty\}$ be upper semicontinuous and bounded from above. 
For $m > 0$ let $F_m := F \vee (-m)$. Since $F_m$ is bounded and upper semicontinuous, the previous step yields
\[
\limsup_{n\rightarrow\infty}\frac{1}{n}\rho_n(nF_m(L_n)) \le \sup_{\nu \in \P_\psi(E)}\left\{F_m(\nu) - \alpha(\nu)\right\} =: S_m,
\]
for each $m > 0$. Since $F_m \ge F$, we have
\[
\rho_n(nF_m(L_n)) \ge \rho_n(nF(L_n))
\]
for each $m$, and it remains only to show that
\begin{align}
\lim_{m \rightarrow \infty}S_m = \sup_{\nu \in \P_\psi(E)}\left\{F(\nu) - \alpha(\nu)\right\} =: S. \label{pf:sanov2}
\end{align}
Clearly $S_m \ge S$, since $F_m \ge F$. Note that $S < \infty$, as $F$ and $\alpha$ are bounded from above and from below, respectively. If $S = -\infty$, then $F(\nu) = -\infty$ whenever $\alpha(\nu) < \infty$, and we conclude that, as $m\rightarrow\infty$,
\[
S_m \le -m - \inf_{\nu \in \P(E)}\alpha(\nu) \ \rightarrow \ -\infty = S.
\]
Now suppose instead that $S$ is finite.
Fix $\epsilon > 0$. For each $m > 0$, find $\nu_m \in \P(E)$ such that
\begin{align}
F_m(\nu_m) - \alpha(\nu_m) + \epsilon \ge S_m \ge S. \label{pf:sanov3}
\end{align}
Since $F$ is bounded from above and $S > -\infty$, it follows that $\sup_m\alpha(\nu_m) <\infty$. The sub-level sets of $\alpha$ are $\psi$-weakly compact, and thus the sequence $(\nu_m)$ has a limit point (in $\P_\psi(E)$). Let $\nu_\infty$ denote any limit point, and suppose $\nu_{m_k} \rightarrow \nu_\infty$.
Note that $\inf_m F_m(\nu_m) > -\infty$ in light of \eqref{pf:sanov3}, because $\alpha$ is bounded from below. Hence, for all sufficiently large $m$, we have $F_m(\nu_m)=F(\nu_m)$. Thus
\begin{align*}
\limsup_{k\rightarrow\infty}\left\{F_{m_k}(\nu_{m_k}) - \alpha(\nu_{m_k})\right\} 	&\le F(\nu_\infty) - \alpha(\nu_\infty)\le S,
\end{align*}
where the second inequality follows from upper semicontinuity of $F$ and lower semicontinuity of $\alpha$.
This holds for any limit point of the pre-compact sequence $(\nu_m)$, and it follows from \eqref{pf:sanov3} that
\[
S \le \limsup_{m\rightarrow\infty}S_m \le \limsup_{m\rightarrow\infty}\left\{F_m(\nu_m) - \alpha(\nu_m)\right\} + \epsilon \le S + \epsilon.
\]
Since $\epsilon > 0$ was arbitrary, this proves \eqref{pf:sanov2}.
\end{proof}

\begin{remark}
Several natural choices of $\alpha$ in fact have sub-level sets which are compact in topology induced by bounded \emph{measurable} test functions, i.e., the topology on $\P(E)$ generated by the maps $\nu \mapsto \int_E f\,d\nu$, where $f \in B(E)$. While this topology is stronger than the usual weak convergence topology, the conclusion of Theorem \ref{th:main-sanov} will likely still hold for bounded functions $F$ which are continuous in this stronger (non-metrizable) topology. This is known to be true in the classical case $\alpha(\cdot)=H(\cdot|\mu)$ (see, e.g. \cite[Section 6.2]{dembozeitouni}), where we recall the definition of relative entropy $H$ from \eqref{def:relativeentropy}. For the sake of brevity, we do not pursue this generalization.
\end{remark}

\subsection{Contraction principles and an abstract form of Cram\'er's theorem} \label{se:contraction}

Viewing Theorem \ref{th:main-sanov-extended} as an abstract form of Sanov's theorem, we may derive from it a form of Cram\'er's theorem. The key tool is an analogue of the contraction principle from classical large deviations (c.f.\ \cite[Theorem 4.2.1]{dembozeitouni}).
In its simplest form, if $\varphi : \P(E) \rightarrow E'$ is continuous for some topological space $E'$, then for $F \in C_b(E')$ we have from Theorem \ref{th:main-sanov}
\begin{align*}
\lim_{n\rightarrow\infty}\frac{1}{n}\rho_n(nF \circ \varphi \circ L_n) &= \sup_{\nu \in \P(E)}\left(F(\varphi(\nu)) - \alpha(\nu)\right) = \sup_{x \in E'}\left(F(x) - \alpha_\varphi(x)\right),
\end{align*}
where we define $\alpha_\varphi : E' \rightarrow (-\infty,\infty]$ by
\[
\alpha_\varphi(x) := \inf\left\{\alpha(\nu) : \nu \in \P(E), \ \varphi(\nu) = x\right\}.
\]
This line of reasoning leads to the following extension of Cram\'er's theorem:

\begin{theorem} \label{th:cramer-abstract}
Let $(E,\|\cdot\|)$ be a separable Banach space with continuous dual $E^*$. Define $\Lambda^* : E \rightarrow \R \cup \{\infty\}$ by
\begin{align*}
\Lambda^*(x) = \sup_{x^* \in E^*}\left(\langle x^*,x\rangle - \rho(x^*)\right).
\end{align*}
Define $S_n : E^n \rightarrow E$ by $S_n(x_1,\ldots,x_n) = \frac{1}{n}\sum_{i=1}^nx_i$.
If $F : E \rightarrow \R \cup \{\infty\}$ is lower semicontinuous and bounded from below, then
\begin{align*}
\liminf_{n\rightarrow\infty}\frac{1}{n}\rho_n(nF \circ S_n) \ge \sup_{x \in E}(F(x)-\Lambda^*(x)).
\end{align*}
Suppose also that the sub-level sets of $\alpha$ are pre-compact subsets of $\P_\psi(E)$, for $\psi(x) := \|x\|$. 
If $F : E \rightarrow \R \cup \{-\infty\}$ is upper semicontinuous and bounded from above, then
\begin{align*}
\limsup_{n\rightarrow\infty}\frac{1}{n}\rho_n(nF \circ S_n) \le \sup_{x \in E}(F(x)-\Lambda^*(x)).
\end{align*}
\end{theorem}

The proof makes use of a proposition, interesting in its own right, which generalizes the well-known result that the functions
\[
t \mapsto \log\int_\R e^{tx}\,\mu(dx), \quad\quad \text{ and } \quad\quad t \mapsto \inf\left\{H(\nu | \mu) : \nu \in \P(\R), \ \int_\R x\,\nu(dx) = t\right\},
\]
are convex conjugates of each other (see, e.g., \cite[Lemma 3.3.3]{dupuis-ellis}). 

\begin{proposition} \label{pr:cramer-representation}
Let $(E,\|\cdot\|)$ be a separable Banach space, and let $\psi(x) = \|x\|$. Suppose the sub-level sets of $\alpha$ are pre-compact subsets of $\P_\psi(E)$. Define $\Psi : E \rightarrow \R \cup \{\infty\}$ by
\[
\Psi(x) = \inf\left\{\alpha(\nu) : \nu \in \P_\psi(E), \ \int_Ez\,\nu(dz) = x\right\},
\]
where the integral is in the sense of Bochner. Define $\Psi^*$ on the continuous dual $E^*$ by
\[
\Psi^*(x^*) = \sup_{x \in E}\left(\langle x^*,x\rangle - \Psi(x)\right).
\]
Then $\Psi$ is convex and lower semicontinuous, and $\Psi^*(x^*) = \rho(x^*)$ for every $x^* \in E^*$. In particular,
\begin{align}
\Psi(x) = \sup_{x^* \in E^*}\left(\langle x^*,x\rangle - \rho(x^*)\right). \label{def:cramer-representation}
\end{align}
\end{proposition}
\begin{proof}
We first show that $\Psi$ is convex. Let $t \in (0,1)$ and $x_1,x_2 \in E$. Fix $\epsilon > 0$, and find $\nu_1,\nu_2 \in \P_\psi(E)$ such that $\int_Ez\nu_i(dz)=x_i$ and $\alpha(\nu_i) \le \Psi(x_i) + \epsilon$. Convexity of $\alpha$ yields
\begin{align*}
\Psi(tx_1 + (1-t)x_2) &\le \alpha(t\nu_1+(1-t)\nu_2) \le t\alpha(\nu_1) + (1-t)\alpha(\nu_2) \\
	&\le t\Psi(x_1) + (1-t)\Psi(x_2) + \epsilon.
\end{align*}
To prove that $\Psi$ is lower semicontinuous, first note that $\Psi$ is bounded from below since $\alpha$ is. Let $x_n \rightarrow x$ in $E$, and find $\nu_n \in \P_\psi(E)$ such that $\alpha(\nu_n) \le \Psi(x_n) + 1/n$ and $\int_Ez\nu_n(dz)=x_n$ for each $n$. Fix a subsequence $\{x_{n_k}\}$ such that $\Psi(x_{n_k}) < \infty$ for all $k$ and $\Psi(x_{n_k})$ converges to a finite value (if no such subsequence exists, then there is nothing to prove, as $\Psi(x_n) \rightarrow \infty$). Then $\sup_{k}\alpha(\nu_{n_k}) < \infty$, and because $\alpha$ has $\psi$-weakly compact sub-level sets there exists a further subsequence (again denoted $n_k$) and some $\nu_\infty \in \P_\psi(E)$ such that $\nu_{n_k}\rightarrow\nu_\infty$.
The convergence $\nu_{n_k}\rightarrow\nu_\infty$ in the $\psi$-weak topology implies
\[
x = \lim_{k\rightarrow\infty}x_{n_k} = \lim_{k\rightarrow\infty}\int_Ez\nu_{n_k}(dz) = \int_Ez\,\nu_\infty(dz).
\]
Using lower semicontinuity of $\alpha$ we conclude
\begin{align}
\Psi(x) \le \alpha(\nu_\infty) \le \liminf_{k\rightarrow\infty}\alpha(\nu_{n_k}) \le \liminf_{k\rightarrow\infty}\Psi(x_{n_k}). \label{pf:representation1}
\end{align}
For every sequence $(x_n)$ in $E$ and any subsequence thereof, this argument shows that there exists a further subsequence for which \eqref{pf:representation1} holds, and this proves that $\Psi$ is lower semicontinuous.
Next, compute $\Psi^*$ as follows:
\begin{align*}
\Psi^*(x^*) &= \sup_{x \in E}\left(\langle x^*,x\rangle  - \Psi(x)\right) \\
	&= \sup_{x \in E}\,\sup\left\{\langle x^*,x\rangle - \alpha(\nu) : \nu \in \P_\psi(E), \ \int_Ez\nu(dz)=x\right\} \\
	&= \sup_{\nu \in \P_\psi(E)}\left(\left\langle x^*,\int_Ez\nu(dz)\right\rangle - \alpha(\nu)\right) \\
	&= \sup_{\nu \in \P_\psi(E)}\left(\int_E\langle x^*,z\rangle\nu(dz) - \alpha(\nu)\right) \\
	&= \rho(x^*).
\end{align*}
Indeed, we can take the supremum equivalently over $\P_\psi(E)$ or over $\P(E)$ in the last two steps, thanks to the assumption that $\alpha = \infty$ off of $\P_\psi(E)$ and our convention $\infty-\infty=-\infty$.
Because $\Psi$ is lower semicontinuous and convex, we conclude from the Fenchel-Moreau theorem \cite[Theorem 2.3.3]{zalinescu-book} that it is equal to its biconjugate, which is precisely what \eqref{def:cramer-representation} says.
\end{proof}

\subsection*{Proof of Theorem \ref{th:cramer-abstract}}
The map
\[
\P_\psi(E) \ni \mu \mapsto F\left(\int_Ez\,\mu(dz)\right)
\]
is upper (resp. lower) semicontinuous as soon as $F$ is upper (resp. lower) semicontinuous. The claims then follow from Theorem \ref{th:main-sanov-extended} and Proposition \ref{pr:cramer-representation}. \hfill\qedsymbol

\section{Compactness of sub-level sets of $\alpha$ in $\P_\psi(E)$} \label{se:psi-compact-levelsets}

Several results of the previous section, such as the upper bound of Theorem \ref{th:main-sanov-extended}, operate under the assumption that the sub-level sets of $\alpha$ are pre-compact subsets of $\P_\psi(E)$. This section compiles some related properties of $(\rho,\alpha)$ which will be useful when we encounter specific examples later in the paper.

\subsection{Cram\'er's condition} \label{se:strongcramer}

A first useful result is a condition under which the effective domain of $\alpha$ is contained in $\P_\psi(E)$.

\begin{proposition} \label{pr:weakcramer}
Fix a measurable function $\psi : E \rightarrow \R_+$. Suppose $\rho(\lambda \psi) < \infty$ for some $\lambda > 0$. Then, for each $\nu \in \P(E)$ satisfying $\alpha(\nu) < \infty$, we have $\int \psi\,d\nu < \infty$.
\end{proposition}
\begin{proof}
By definition, for each $\nu \in \P(E)$,
\[
\infty > \rho(\lambda\psi) \ge \lambda\int \psi\,d\nu - \alpha(\nu).
\]
If $\alpha(\nu) < \infty$ then certainly $\int\psi\,d\nu < \infty$.
\end{proof}

The next and more important proposition identifies a condition under which the sub-level sets of $\alpha$ are not only weakly compact (which we always assume) but also $\psi$-weakly compact.

\begin{proposition} \label{pr:strongcramer}
Fix a continuous function $\psi : E \rightarrow \R_+$. Suppose
\begin{align}
\lim_{m\rightarrow\infty}\rho(\lambda\psi 1_{\{\psi \ge m\}}) = \rho(0), \ \forall \lambda > 0. \label{def:strongcramer-condition}
\end{align}
Then, for each $c \in \R$, the weak and $\psi$-weak topologies coincide on $\{\nu \in \P(E) : \alpha(\nu) \le c\} \subset \P_\psi(E)$; in particular, the sub-level sets of $\alpha$ are $\psi$-weakly compact.
\end{proposition}

A first step in the proof comes from the following simple lemma, worth stating separately for emphasis:

\begin{lemma} \label{le:strongcramer-implies-weak}
Fix a continuous function $\psi : E \rightarrow \R_+$. Suppose \eqref{def:strongcramer-condition} holds. 
Then $\rho(\lambda\psi) < \infty$ for every $\lambda \ge 0$. In particular, for each $\nu \in \P(E)$ satisfying $\alpha(\nu) < \infty$, we have $\int\psi\,d\nu < \infty$.
\end{lemma}
\begin{proof}
The second claim is just Proposition \ref{pr:weakcramer}.
For $m,\lambda > 0$ we have  $\lambda\psi \le \lambda m + \lambda\psi1_{\{\psi \ge m\}}$, and thus properties (R1) and (R2) of Theorem \ref{th:alpha-to-rho} imply
\[
\rho(\lambda \psi) \le \lambda m + \rho(\lambda\psi 1_{\{\psi \ge m\}}).
\]
By \eqref{def:strongcramer-condition}, for $m$ sufficiently large the right-hand side is finite.
\end{proof}

\begin{proof}[Proof of Proposition \ref{pr:strongcramer}]
Fix $c \in \R$, and abbreviate $S = \{\nu \in \P(E) : \alpha(\nu) \le c\}$. Assume $S \neq \emptyset$. Note that Lemma \ref{le:strongcramer-implies-weak} implies $S \subset \P_\psi(E)$.
It suffices to prove that the map $\nu \mapsto \int_Ef\,d\nu$ is weakly continuous on $S$ for every continuous $f : E \to \R$ with $|f| \le 1 + \psi$. 
Note that for $\eta_n,\eta \in \P(\R)$ with $\eta_n\to\eta$ weakly we have $\int g\,d\eta_n \to \int g\,d\eta$ for each continuous function $g$ which is uniformly integrable in the sense that
\[
\lim_{m\rightarrow\infty}\sup_n\int_{\{|g| \ge m\}}|g|\,d\eta_n = 0.
\]
(See \cite[Theorem A.3.19]{dupuis-ellis}.) Applying this to the image measures $\{\nu \circ f^{-1} : \nu \in S\}$ for $f$ as above, we find that it suffices
to prove the uniform integrability condition
\[
\lim_{m\rightarrow\infty}\sup_{\nu \in S}\int_{\{\psi \ge m\}}\psi\,d\nu = 0.
\]
By definition of $\rho$, for $m > 0$ and $\nu \in S$,
\begin{align}
\lambda\int_{\{\psi \ge m\}}\psi\,d\nu &\le \rho(\lambda\psi 1_{\{\psi \ge m\}}) + \alpha(\nu) \le \rho(\lambda\psi 1_{\{\psi \ge m\}}) + c, \label{pf:strongcramer1}
\end{align}
Given $\epsilon > 0$, choose $\lambda > 0$ large enough that $(\epsilon + \rho(0) + c)/\lambda \le \epsilon$. Then choose $m$ large enough that $\rho(\lambda\psi 1_{\{\psi \ge m\}}) \le \epsilon + \rho(0)$, which is possible because of assumption \eqref{def:strongcramer-condition}. It then follows from \eqref{pf:strongcramer1} that $\int_{\{\psi \ge m\}}\psi\,d\nu \le \epsilon$, and the proof is complete.
\end{proof}

We refer to \eqref{def:strongcramer-condition} as the \emph{strong Cram\'er condition}.
Several extensions of the classical form of Sanov's theorem to stronger topologies rely on what might be called a ``strong Cram\'er condition.'' For instance, if $\psi : E \rightarrow \R_+$ is continuous, the results of Schied \cite{schied1998cramer} indicate that Sanov's theorem can be extended to the $\psi$-weak topology if (and essentially only if) $\log\int_Ee^{\lambda\psi}\,d\mu < \infty$ for every $\lambda \ge 0$; see also \cite{wang2010sanov,eichelsbacher1996large}.

The form of our strong Cram\'er condition \eqref{def:strongcramer-condition} was heavily inspired by the work of Owari \cite{owari2014maximum} on continuous extensions of monotone convex functionals.
In several cases of interest (namely, Propositions \ref{pr:shortfallcramercondition} and \ref{pr:robustcramercondition} below), it turns out that a converse to Lemma \ref{le:strongcramer-implies-weak} is true, i.e., the strong Cram\'er condition \eqref{def:strongcramer-condition} is equivalent to the statement that $\rho(\lambda\psi) < \infty$ for all $\lambda > 0$. In general, however, the strong Cram\'er condition is the strictly stronger statement.
Consider the following simple example, borrowed from \cite[Example 3.7]{owari2014maximum}: Let $E = \{0,1,\ldots,\}$ be the natural numbers, and define $\mu_n \in \P(E)$ by $\mu_1\{0\}=1$, $\mu_n\{0\} = 1-1/n$, and $\mu_n\{n\} = 1/n$. Let $M$ denote the closed convex hull of $(\mu_n)$. Then $M$ is convex and weakly compact. Define $\alpha(\mu) = 0$ for $\mu \in M$ and $\alpha(\mu)=\infty$ otherwise. Then $\alpha$ satisfies our standing assumptions, and $\rho(f) = \sup_{\mu \in M}\int f\,d\mu = \sup_n\int f\,d\mu_n$. Finally, let $\psi(x)=x$ for $x \in E$. Then $\rho(\lambda\psi) = \lambda < \infty$ because $\int \psi\,d\mu_n = 1$ for all $n$, and similarly $\rho(\lambda \psi1_{\{\psi \ge m\}}) = \lambda$ because $\int\psi1_{\{\psi \ge m\}}\,d\mu_n = 1_{\{n \ge m\}}$. In particular, $\rho(\lambda\psi) < \infty$ for all $\lambda > 0$, but the strong Cram\'er condition fails.

Finally, we remark that it is conceivable that a converse to Proposition \ref{pr:strongcramer} might hold, i.e., that the strong Cram\'er condition \eqref{def:strongcramer-condition} may be \emph{equivalent} to the pre-compactness of the sub-level sets of $\alpha$ in $\P_\psi(E)$. Indeed, the results of Schied \cite[Theorem 2]{schied1998cramer} and Owari \cite[Theorem 3.8]{owari2014maximum} suggest that this may be the case. But this remains an open problem.

\subsection{Implications of $\psi$-weakly compact sub-level sets}
This section contains two results to be used occasionally in the sequel. First is a useful lemma that aid in the computation of $\rho(f)$ for certain unbounded $f$ in Section \ref{se:nonexpLDP}.

\begin{lemma} \label{le:extended-duality2}
If $f : E \to \R$ is upper semicontinuous and bounded from above, then
\begin{align*}
\rho(f) = \lim_{m\to\infty}\rho(f \vee (-m)) = \inf_{m \ge 0}\rho(f \vee (-m)).
\end{align*}
If $f : E \to \R$ is measurable and bounded from below, then
\begin{align*}
\rho(f) = \lim_{m\to\infty}\rho(f \wedge m) = \sup_{m \ge 0}\rho(f \wedge m).
\end{align*}
Lastly, let $\psi : E \rightarrow \R_+$ be continuous. If the sub-level sets of $\alpha$ are pre-compact subsets of $\P_\psi(E)$, and if $f : E \to \R$ is measurable with $f \ge -c(1+\psi)$ pointwise for some $c \ge 0$, then
\begin{align*}
\rho(f) = \lim_{m\to\infty}\rho(f \wedge m) = \sup_{m \ge 0}\rho(f \wedge m).
\end{align*}
\end{lemma}
\begin{proof}
The second claim is a special case of the final claim with $\psi \equiv 0$.
To prove the final claim, note first that $\rho(f \wedge m)$ is non-decreasing in $m$ (see (R1) of Theorem \ref{th:alpha-to-rho}). We find
\begin{align*}
\rho(f) &= \sup_{\nu \in \P_\psi(E)}\left(\int_E f\,d\nu - \alpha(\nu)\right) = \sup_{m \ge 0}\sup_{\nu \in \P_\psi(E)}\left(\int_E f \wedge m\,d\nu - \alpha(\nu)\right) \\
	&= \sup_{m \ge 0}\rho(f \wedge m).
\end{align*}
Indeed, for each $\nu \in \P_\psi(E)$, the monotone convergence theorem applies because $f \wedge m$ for $m \ge 0$ are bounded from below by the $\nu$-integrable function $-c(1+\psi)$.
To prove the first claim, abbreviate $f_m= f \vee (-m)$ for $m \ge 0$.
Monotonicity of $\rho$ implies $\inf_{m \ge 0}\rho(f_m) \ge \rho(f)$, so we need only prove the reverse inequality. Assume without loss of generality that $\inf_{m \ge 0}\rho(f_m) > -\infty$.
For each $n$, we may find for each $n$ some $\nu_n \in \P_\psi(E)$ such that
\begin{align}
-\infty < \inf_{m \ge 0}\rho(f_m)  &\le \rho(f_n) \le \int_Ef_n\,d\nu_n - \alpha(\nu_n) + 1/n. \label{pf:extended-duality1}
\end{align}
This implies $\sup_n\alpha(\nu_n) < \infty$, because $f$ is bounded from above. Pre-compactness of the sub-level sets of $\alpha$ allows us to extract a subsequence ${n_k}$ and $\nu_\infty \in \P(E)$ such that $\nu_{n_k} \rightarrow \nu_\infty$ weakly. 
By Skorohod's representation, we may construct random variables $X_k$ and $X_\infty$ with respective laws $\nu_{n_k}$ and $\nu_\infty$ such that $X_k \rightarrow X_\infty$ a.s. 
The upper semicontinuity assumption implies $\limsup_{k\rightarrow\infty}f_{n_k}(X_k) \le f(X_\infty)$ almost surely. We then conclude from Fatou's lemma that
\[
\limsup_{k\rightarrow\infty}\int_Ef_{n_k}\,d\nu_{n_k} = \limsup_{k\rightarrow\infty}\E[f_{n_k}(X_k)] \le \E[f(X_\infty)] = \int_Ef\,d\nu_\infty.
\]
Since $\alpha$ is weakly lower semicontinuous, we conclude from \eqref{pf:extended-duality1} that
\[
\inf_{m \ge 0}\rho(f_m) \le \int_Ef\,d\nu_\infty - \alpha(\nu_\infty) \le \sup_{\nu \in \P(E)}\left(\int_Ef\,d\nu - \alpha(\nu)\right) = \rho(f).
\]
\end{proof}

\section{Non-exponential large deviations} \label{se:nonexpLDP}

The goal of this section is to prove Theorem \ref{th:lpsanov} and its consequences detailed in Section \ref{se:intro:nonexpLDP}, but along the way we will explore a particularly interesting class of $(\alpha,\rho)$ pairs.

\subsection{Shortfall risk measures} \label{se:shortfall}
Fix $\mu \in \P(E)$ and a nondecreasing, nonconstant, convex function $\ell : \R \rightarrow \R_+$ satisfying $\ell(x) < 1$ for all $x < 0$. Let $\ell^*(y) = \sup_{x \in \R}(xy - \ell(x))$ denote the convex conjugate, and define $\alpha : \P(E) \rightarrow [0,\infty]$ by
\begin{align}
\alpha(\nu) = \begin{cases}
\inf_{t > 0}\frac{1}{t}\left(1 + \int_E\ell^*\left(t\frac{d\nu}{d\mu}\right)\,d\mu\right) &\text{if } \nu \ll \mu \\
\infty &\text{otherwise}.
\end{cases}  \label{def:shortfall-alpha}
\end{align}
Note that $\ell^*(x) \ge - \ell(0) \ge -1$, by assumption and by continuity of $\ell$, so that $\alpha \ge 0$.
Define $\rho$ as usual by \eqref{intro:duality}. It is known \cite[Proposition 4.115]{follmer-schied-book} that, for $f \in B(E)$,
\begin{align}
\rho(f) = \inf\left\{m \in \R : \int_E\ell(f(x)-m)\mu(dx) \le 1\right\}. \label{def:shortfall-rho}
\end{align}
Refer to the book of F\"ollmer and Schied \cite[Section 4.9]{follmer-schied-book} for a thorough study of the properties of $\rho$. Notably, they show that $\rho$ satisfies all of properties (R1--4) of Theorem \ref{th:alpha-to-rho}, and that both dual formulas hold:
\begin{align*}
\rho(f) &= \sup_{\nu \in \P(E)}\left(\int_Ef\,d\nu - \alpha(\nu)\right), \quad\quad
\alpha(\nu) = \sup_{f \in B(E)}\left(\int_Ef\,d\nu - \rho(f)\right).
\end{align*}
If $\ell(x)=e^x$ we recover $\rho(f)=\log\int_Ee^f\,d\mu$ and $\alpha(\nu) = H(\nu | \mu)$. If $\ell(x) = [(1+x)^+]^q$ for some $q \ge 1$, then
\begin{align}
\alpha(\nu) = \|d\nu/d\mu\|_{L^p(\mu)}-1, \text{ for } \nu \ll \mu, \quad\quad \alpha(\nu) = \infty \text{ otherwise}, \label{def:lpentropy}
\end{align}
where $p=q/(q-1)$, and where of course $\|f\|_{L^p(\mu)} = \left(\int |f|^p\,d\mu\right)^{1/p}$; see \cite[Example 4.118]{follmer-schied-book} or \cite[Section 3.1]{lacker-liquidity} for this computation. The $-1$ is a convenient normalization, ensuring that $\alpha(\nu)=0$ if and only if $\nu=\mu$.


We work in the rest of this subsection with $\alpha$ and $\rho$ given as in \eqref{def:shortfall-alpha} and \eqref{def:shortfall-rho}.
The following result shows how the strong Cram\'er condition \eqref{def:strongcramer-condition} simplifies in the present context. It is essentially contained in \cite[Proposition 7.3]{owari2014maximum}, but we include the short proof.

\begin{proposition} \label{pr:shortfallcramercondition}
Let $\psi : E \rightarrow \R_+$ be measurable. Suppose $\int_E\ell(\lambda\psi(x))\mu(dx) < \infty$ for all $\lambda > 0$. Then the strong Cram\'er condition holds, $\lim_{m\rightarrow\infty}\rho(\lambda\psi 1{\{\psi \ge m\}})\rightarrow 0$ for each $\lambda > 0$. In particular, the sub-level sets of $\alpha$ are compact subsets of $\P_\psi(E)$.
\end{proposition}
\begin{proof}
The final claim is simply an application of Proposition \ref{pr:strongcramer}.
Fix $\epsilon > 0$ and $\lambda > 0$. Since $\ell$ is nondecreasing, the following two limits hold:
\begin{align*}
\lim_{m\rightarrow\infty}\mu(\psi < m) = 1, \quad\quad\quad \lim_{m\rightarrow\infty}\int_{\{\psi \ge m\}}\ell\left(\lambda\psi(x) - \epsilon\right)\mu(dx) = 0.
\end{align*}
Since $\ell(-\epsilon) < 1$, it follows that, for sufficiently large $m$,
\begin{align*}
1 &\ge \ell(-\epsilon)\mu(\psi < m) + \int_{\{\psi \ge m\}}\ell\left(\lambda\psi(x) - \epsilon\right)\mu(dx) \\
	&= \int_E\ell\left(\lambda\psi(x)1_{\{\psi \ge m\}}(x) - \epsilon\right)\mu(dx).
\end{align*}
Next, the second assertion of Lemma \ref{le:extended-duality2} implies $\rho(\lambda\psi 1_{\{\psi \ge m\}}) = \sup_{n\ge 0}\rho(n \wedge (\lambda\psi 1_{\{\psi \ge m\}}))$. For each $n$, we use the identity \eqref{def:shortfall-rho} which is valid for bounded $f$ to get, for sufficiently large $m$,
\begin{align*}
\rho(\lambda\psi 1_{\{\psi \ge m\}}) &=  \sup_{n\ge 0}\rho(n \wedge (\lambda\psi 1_{\{\psi \ge m\}})) \\
	&= \sup_{n \ge 0}\inf\left\{c \in \R : \int_E\ell\left(n \wedge (\lambda\psi 1_{\{\psi \ge m\}}) - c\right)\,d\mu \le 1 \right\} \\
	&\le \inf\left\{c \in \R : \int_E\ell\left( \lambda\psi 1_{\{\psi \ge m\}} - c\right)\,d\mu \le 1 \right\} \\
	&\le \epsilon.
\end{align*}
\end{proof}

%

Note that \eqref{def:shortfall-rho} is only valid, a priori, for bounded $f$, although the expression on the right-hand side certainly makes sense for unbounded $f$. The next results provide some useful cases for which the identity \eqref{def:shortfall-rho} carries over to unbounded functions, and these will be needed in the proof of Corollary \ref{co:lpcramer}. In the following, define $\ell(\pm \infty) = \lim_{x \rightarrow \pm\infty}\ell(x)$.

\begin{proposition} \label{pr:shortfall-L1}
Let $\psi : E \rightarrow \R_+$ be continuous, and suppose $\int_E\ell(\lambda\psi(x))\mu(dx) < \infty$ for all $\lambda > 0$. Suppose $f : E \rightarrow \R$ is continuous with $|f| \le c(1+\psi)$ pointwise for some $c \ge 0$. Then the identity \eqref{def:shortfall-rho} holds.
\end{proposition}
\begin{proof}
Let $H(f)$ denote the right-hand side of \eqref{def:shortfall-rho}, well defined for any measurable function $f : E \to \R$. We must show $\rho(f)=H(f)$ for $f$ as in the statement of the proposition. As was mentioned above, it is known from \cite[Proposition 4.115]{follmer-schied-book} that $\rho(f)=H(f)$ for bounded $f$. \\

\noindent\textbf{Step 1.}
Assume first that $f$ is continuous and bounded from above, with $|f| \le c(1+\psi)$. Let $f_n = f \vee (-n)$ for $n \ge 0$. Since $f_n$ is bounded for each $n$, we have $\rho(f_n)=H(f_n)$. The first assertion of Lemma \ref{le:extended-duality2} then implies $\rho(f) = \lim_{n\to\infty}\rho(f_n) = \lim_{n\to\infty}H(f_n)$. It remains to show $H(f_n)\to H(f)$. Clearly $H(f_n) \ge H(f_{n+1}) \ge H(f)$ for each $n$ since $f_n \ge f_{n+1}$ pointwise and $\ell$ is nondecreasing, so the sequence $H(f_n)$ has a limit. As $\ell$ is continuous and strictly increasing in a neighborhood of the origin, note that $H(f)$ is the unique solution $c \in \R$ of the equation
\begin{align}
\int_E\ell(f(x)-c)\mu(dx) = 1. \label{pf:shortfall-L1-11}
\end{align}
Similarly, $H(f_n)$ uniquely solves $\int_E\ell(f_n(x)-H(f_n))\mu(dx) = 1$. Let $c = \lim_{n\to\infty}H(f_n)$, and note that the integrands $\ell(f_n(x)-H(f_n))$ are uniformly bounded and converge pointwise to $\ell(f(x)-c)$. Passing to the limit using dominated convergence shows that $c$ solves the equation \eqref{pf:shortfall-L1-11}, which implies $c=H(f)$. \\

\noindent\textbf{Step 2.}
We now turn to general continuous $f$ satisfying $|f| \le c(1+\psi)$. Define $f_n = f \wedge n$ for $n \ge 0$, so that $f_n$ is bounded from above. By Step 2, $\rho(f_n)=H(f_n)$ for each $n$. By Proposition \ref{pr:shortfallcramercondition}, the sub-level sets of $\alpha$ are $\psi$-weakly compact, and the third assertion of Lemma \ref{le:extended-duality2} yields $\rho(f_n) \to \rho(f)$. It remains to show $H(f_n)\to H(f)$. To see this, note first that $H(f_n) \le H(f_{n+1}) \le H(f)$ for each $n$ since $f_n \le f_{n+1}$ pointwise. Let $\epsilon > 0$ and $c = H(f)-\epsilon$, and note that the definition of $H$ and monotonicity of $\ell$ imply $\int_E\ell(f(x) - c)\mu(dx) > 1$. By monotone convergence, there exists $n$ such that $\int_E\ell(f_n(x) - c)\mu(dx) > 1$. The definition of $H$ now implies $H(f_n) > c = H(f) - \epsilon$. As $\epsilon > 0$ was arbitrary, we conclude $H(f_n)\to H(f)$.
\end{proof}

We record here for later use a simple but useful lemma:
\begin{lemma} \label{le:alphastrictlypositive}
Define $\alpha$ as in \eqref{def:lpentropy}.
Let $\psi : E \rightarrow \R_+$ be continuous, and suppose $\int_E \psi^q\,d\mu < \infty$.
Suppose $A \subset \P_\psi(E)$ is closed (in the $\psi$-weak topology), and $\mu \notin A$. Then $\inf_{\nu \in A}\alpha(\nu) > 0$.
\end{lemma}
\begin{proof}
Recall that $\alpha$ as in \eqref{def:lpentropy} is the special case of \eqref{def:shortfall-alpha} corresponding to $\ell(x) = [(1+x)^+]^q$. Thus Proposition \ref{pr:shortfallcramercondition} and the assumption $\int_E \psi^q\,d\mu < \infty$ ensure that the sub-level sets of $\alpha$ are $\psi$-weakly compact.
If $\inf_{\nu \in A}\alpha(\nu) =0$, we may find $\nu_n \in A$ such that $\alpha(\nu_n) \rightarrow 0$. The sequence $(\nu_n)$ admits a $\psi$-weak limit point $\nu^*$, which must of course belong to the $\psi$-weakly closed set $A$. Lower semicontinuity and nonnegativity of $\alpha$ imply $\alpha(\nu^*) = 0$. This implies $\nu^*=\mu$, as $t \mapsto t^p$ is strictly convex, and this contradicts the assumption that $\mu \notin A$.
\end{proof}

\subsection{Proofs of Theorem \ref{th:lpsanov}, Corollary \ref{co:wasserstein-rate}, and Corollary \ref{co:lpcramer}} \label{se:pf:lpstuff}

With these generalities in hand, we now turn toward the proof of Theorem \ref{th:lpsanov}. The idea is to apply Theorem \ref{th:main-sanov-extended} with $\alpha$ defined as in \eqref{def:lpentropy}. The following estimate is crucial:

\begin{lemma} \label{le:shortfallbound}
Let $q \in (1,\infty]$, and let $p=q/(q-1)$ denote the conjugate exponent. Let $\alpha$ be as in \eqref{def:lpentropy}. Then, for each $n \ge 1$ and $\nu \in \P(E^n)$ with $\nu \ll \mu^n$,
\begin{align}
\alpha_n(\nu) \le n^{1/q}\|d\nu/d\mu^n\|_{L^p(\mu^n)}. \label{pf:lpsanov3}
\end{align}
\end{lemma}
\begin{proof}
The case $p=\infty$ and $q=1$ follows by sending $p \rightarrow\infty$ in \eqref{pf:lpsanov3}, so we prove only the case $p < \infty$.
As we will be working with conditional expectations, it is convenient to work with a more probabilistic notation:
Fix $n$, and endow $\Omega = E^n$ with its Borel $\sigma$-field as well as the probability $P = \mu^n$. Let $X_i : E^n \rightarrow E$ denote the natural projections, and let $\F_k = \sigma(X_1,\ldots,X_k)$ denote the natural filtration, for $k=1,\ldots,n$, with $\F_0 := \{\emptyset,\Omega\}$. For $\nu \in \P(E^n)$ and $k=1,\ldots,n$, let $\nu_k$ denote a version of the regular conditional law of $X_k$ given $\F_{k-1}$ under $\nu$, or symbolically $\nu_k := \nu(X_k \in \cdot \, | \, \F_{k-1})$. Let $\E^\nu$ denote integration with respect to $\nu$. Since $P(X_k \in \cdot \, | \, \F_{k-1}) = \mu$ a.s., if $\nu \ll P$ then
\[
\frac{d\nu_k}{d\mu} = \frac{\E^P[d\nu/dP | \F_k]}{\E^P[d\nu/dP | \F_{k-1}]} =: \frac{M_k}{M_{k-1}}, \text{ a.s., where } \frac{0}{0} := 0.
\]
Therefore
\[
\alpha(\nu_k) = \E^P\left[\left.\left(\frac{M_k}{M_{k-1}}\right)^p\right|\F_{k-1}\right]^{1/p}-1.
\]
Note that $(M_k)_{k=0}^n$ is a nonnegative martingale, with $M_0 = 1$ and $M_n = d\nu/dP$. Then
\begin{align*}
\alpha_n(\nu) &= \E^\nu\left[\sum_{k=1}^n\alpha(\nu_k)\right] = \E^P\left[M_n\sum_{k=1}^n\left(\E^P\left[\left.\left(\frac{M_k}{M_{k-1}}\right)^p\right|\F_{k-1}\right]^{1/p}-1\right)\right] \\
	&= \E^P\left[\sum_{k=1}^n\left(\E^P\left[\left.M_k^p \right|\F_{k-1}\right]^{1/p}-M_{k-1}\right)\right].
\end{align*}
Subadditivity of $x \mapsto x^{1/p}$ implies
\[
\left(\E^P[M_k^p|\F_{k-1}]\right)^{1/p} \le \left(\E^P[M_k^p - M_{k-1}^p|\F_{k-1}]\right)^{1/p} + M_{k-1},
\]
where the right-hand side is well-defined because
\[
\E^P[M_k^p | \F_{k-1}] \ge\E^P[M_k | \F_{k-1}]^p = M_{k-1}^p.
\]
Concavity of $x \mapsto x^{1/p}$ and Jensen's inequality yield
\begin{align*}
\alpha_n(\nu) &\le \E^P\left[\sum_{k=1}^n\left(\E^P[M_k^p - M_{k-1}^p|\F_{k-1}]\right)^{1/p}\right] \\
	&\le n^{1-\frac{1}{p}}\left(\E^P\left[\sum_{k=1}^n\E^P[M_k^p - M_{k-1}^p|\F_{k-1}]\right]\right)^{1/p} \\
	&= n^{1/q}\left(\E^P\left[M_n^p - M_0^p\right]\right)^{1/p} \\
	&\le n^{1/q}\left(\E^P\left[M_n^p\right]\right)^{1/p}.
\end{align*}
\end{proof}

\subsection*{Proof of Theorem \ref{th:lpsanov}}
Again, let $q \in (1,\infty)$ and $p=q/(q-1)$, and let $\alpha$ be as in \eqref{def:lpentropy}, noting that it corresponds to \eqref{def:shortfall-alpha} with $\ell(x) = [(1+x)^+]^q$. Then Proposition \ref{pr:shortfallcramercondition} and the assumption that $\int\psi^q\,d\mu < \infty$ imply that the sub-level sets of $\alpha$ are pre-compact subsets of $\P_\psi(E)$. Hence, Theorem \ref{th:main-sanov-extended} applies to the $\psi$-weakly upper semicontinuous function $F : \P_\psi(E) \rightarrow [-\infty,0]$ defined by $F(\nu) = 0$ if $\nu \in A$ and $F(\nu) = -\infty$ otherwise. This yields
\begin{align}
\limsup_{n\rightarrow\infty}\frac{1}{n}\rho_n(nF \circ L_n) \le -\inf_{\nu \in A}\alpha(\nu). \label{pf:lpsanov1.1}
\end{align}
Now use Lemma \ref{le:shortfallbound}, noting that $\frac{1}{n}n^{1/q} = n^{-1/p}$, to get
\begin{align*}
\frac{1}{n}\rho_n(nF\circ L_n) &= \sup_{\nu \in \P(E^n)}\left(\int_{E^n} F\circ L_n\,d\nu- \frac{1}{n}\alpha_n(\nu)\right) \\
	&= -\inf\left\{\frac{1}{n}\alpha_n(\nu) : \nu \in \P(E^n), \ \nu(L_n \in A)=1\right\} \\
	&\ge -\inf\left\{n^{-1/p}\|d\nu/d\mu^n\|_{L^p(\mu^n)}  : \nu \in \P(E^n), \ \nu \ll \mu^n, \ \nu(L_n \in A)=1\right\}.
\end{align*}
Set $B_n = \{x \in E^n : L_n(x) \in A\}$, and define $\nu \ll \mu^n$ by $d\nu/d\mu^n = 1_{B_n}/\mu^n(B_n)$. A quick computation yields
\[
\|d\nu/d\mu^n\|_{L^p(\mu^n)} = \mu^n(B_n)^{(1-p)/p} = \mu^n(B_n)^{-1/q}.
\]
Thus
\[
\frac{1}{n}\rho_n(nF \circ L_n) \ge -\left(n^{1/p}\mu^n(B_n)^{1/q}\right)^{-1}.
\]
Combine this with \eqref{pf:lpsanov1.1} to get
\begin{align*}
\limsup_{n\rightarrow\infty}-\left(n^{1/p}\mu^n(L_n \in A)^{1/q}\right)^{-1} \le -\inf_{\nu \in A}\alpha(\nu).
\end{align*}
Recalling the definition of $\alpha$ from \eqref{def:lpentropy} and noting that $q/p = q-1$, this inequality rewrites as the desired result.
\hfill \qedsymbol

\subsection*{Proof of Corollary \ref{co:wasserstein-rate}}
Define a continuous function $\psi : E \to \R_+$ by $\psi(x) = d^r(x,x_0)$. Note that $\W_r$ then metrizes $\P_\psi(E)$ (see \cite[Theorem 7.12]{villani-book}). Hence, the set
\[
A = \left\{\nu \in \P_\psi(E) : \W_r(\nu,\mu) \ge a\right\}
\]
is closed in $\P_\psi(E)$. Because $\int \psi^{q/r}\,d\mu = \int d^q(x,x_0)\,\mu(dx) < \infty$ by assumption, we may apply Theorem \ref{th:lpsanov} with $q/r$ in place of $q$ to get 
\begin{align*}
\limsup_{n\to\infty} n^{\frac{q}{r}-1}\mu^n\left(\W_r(L_n,\mu) \ge a\right) \le \left(\inf_{\nu \in A} \alpha(\nu)\right)^{-q/r},
\end{align*}
where $\alpha$ is defined as in \eqref{def:lpentropy} with $p=(q/r)/(q/r-1)$. It remains to show that $\inf_{\nu \in A} \alpha(\nu) > 0$. But this follows from Lemma \ref{le:alphastrictlypositive}, since $A$ is closed, $\mu \notin A$, and $\int \psi^{q/r}\,d\mu < \infty$.
\hfill \qedsymbol

\subsection*{Proof of Corollary \ref{co:lpcramer}}
Again, let $\alpha$ be as in \eqref{def:lpentropy}, and note that it corresponds to the shortfall risk measure \eqref{def:shortfall-rho} with $\ell(x) = [(1+x)^+]^q$.
Let $\psi(x) = \|x\|$, and consider the $\P_\psi(E)$-closed set
\[
B = \left\{\mu \in \P_\psi(E) : \int_Ez\,\mu(dz) \in A\right\},
\]
where the integral is defined in the Bochner sense. Proposition \ref{pr:shortfallcramercondition} and the assumption that $\int\psi^q\,d\mu = \E[\|X_1\|^q] < \infty$ imply that the sub-level sets of $\alpha$ are pre-compact subsets of $\P_\psi(E)$.
We may then apply Theorem \ref{th:lpsanov} to get
\[
\limsup_{n\rightarrow\infty}n^{q-1}\PP\left(\frac{1}{n}\sum_{i=1}^nX_i \in A\right) \le \left(\inf_{\nu \in B}\alpha(\nu)\right)^{-q},
\]
where again $\alpha$ is as in \eqref{def:lpentropy}.
It remains to simplify the right-hand side. Proposition \ref{pr:cramer-representation} yields
\[
\sup_{x^* \in E^*}(\langle x^*,x\rangle - \rho(x^*)) = \inf\left\{\alpha(\nu) : \nu \in \P_\psi(E), \ \int_Ez\,\nu(dz)=x\right\}, \text{ for } x \in E.
\]
Infimize over $x \in A$ on both sides to get
\begin{align}
\inf_{\nu \in B}\alpha(\nu) = \inf_{x \in A}\sup_{x^* \in E^*}(\langle x^*,x\rangle - \rho(x^*)). \label{pf:cramer1111}
\end{align}
According to Proposition \ref{pr:shortfall-L1}, for $x^* \in E^*$ we have
\begin{align*}
\rho(x^*) &= \inf\left\{m \in \R : \int_E [(1+x^*(x)-m)^+]^q\mu(dx) \le 1\right\} = \Lambda(x^*),
\end{align*}
where the latter equality is simply the definition of $\Lambda$ given in the statement of Corollary \ref{co:lpcramer}. Hence, the identity \eqref{pf:cramer1111} becomes $\inf_{\nu \in B}\alpha(\nu) = \inf_{x \in A}\Lambda^*(x)$, and the proof is complete.
\hfill \qedsymbol

\subsection{Stochastic optimization with heavy tails} \label{se:optimization}
This section elaborates on the application discussed in Section \ref{se:intro:optimization}, concerning the convergence of Monte-Carlo estimates for stochastic optimization problems. We use here the notation of Section \ref{se:intro:optimization}, and we begin with the proof of Theorem \ref{th:valueconvergence}:

\subsection*{Proof of Theorem \ref{th:valueconvergence}}
Let $A = \{\nu \in \P_\psi(E) : |V(\nu)-V(\mu)| \ge \epsilon\}$. The map
\[
\X \times \P_\psi(E) \ni (x,\nu) \mapsto \int_Eh(x,w)\nu(dx)
\]
is jointly continuous. By Berge's theorem \cite[Theorem 17.31]{aliprantisborder}, $V$ is continuous on $\P_\psi(E)$, and so $A$ is closed.
Theorem \ref{th:lpsanov} implies 
\begin{align*}
\limsup_{n\rightarrow\infty}n^{q-1}\mu^n(|V(L_n)-V(\mu)| \ge \epsilon) &= \limsup_{n\rightarrow\infty}n^{q-1}\mu^n(L_n \in A) \le \left(\inf_{\nu \in A}\alpha(\nu)\right)^{-q}.
\end{align*}
Note that $q/p=q-1$, and finally use Lemma \ref{le:alphastrictlypositive} to conclude $\inf_{\nu \in A}\alpha(\nu) > 0$.
\hfill\qedsymbol

\begin{remark}
The joint continuity and compactness assumptions in Theorem \ref{th:valueconvergence} could likely be weakened, but we focus on the more novel integrability issues to ease the exposition.
\end{remark}

Now that we have shown the optimal value itself converges, we turn next to the convergence of optimizers themselves.

\begin{theorem} \label{th:optimizerconvergence}
Grant the assumptions of Theorem \ref{th:valueconvergence}.
Let $\hat{x} : \P_\psi(E) \rightarrow \X$ be any measurable function satisfying\footnote{Such a function $\hat{x}$ exists because $(x,\nu) \mapsto \int_Eh(x,w)\nu(dw)$ is measurable in $\nu$ and continuous in $x$; see, e.g., \cite[Theorem 18.19]{aliprantisborder}.}
\[
\hat{x}(\nu) \in \arg\min_{x \in \X}\int_Eh(x,w)\nu(dw), \text{ for each } \nu.
\]
Suppose there exist a measurable function $\varphi : \R \rightarrow \R$ and a compatible metric $d$ on $\X$ such that
\[
\varphi(d(\hat{x}(\mu),x)) \le \int_Eh(x,w)\mu(dw) - \int_Eh(\hat{x}(\mu),w)\mu(dw).
\]
Then, for any $\epsilon > 0$,
\[
\limsup_{n\rightarrow\infty}n^{q-1}\mu^n(\varphi(d(\hat{x}(\mu),\hat{x}(L_n))) \ge \epsilon) < \infty.
\]
In particular, if $\varphi$ is strictly increasing with $\varphi(0)=0$,  then for any $\epsilon > 0$,
\[
\limsup_{n\rightarrow\infty}n^{q-1}\mu^n(d(\hat{x}(\mu),\hat{x}(L_n)) \ge \epsilon) < \infty.
\]
\end{theorem}
\begin{proof}
Note that for $\epsilon > 0$, on the event $\{\varphi(d(\hat{x}(\mu),\hat{x}(L_n))) \ge \epsilon\}$ we have
\begin{align*}
\epsilon &\le \varphi(d(\hat{x}(\mu),\hat{x}(L_n))) \le \int_Eh(\hat{x}(L_n),w)\mu(dw) - \int_Eh(\hat{x}(\mu),w)\mu(dw) \\
	&\le |V(L_n)-V(\mu)| + \sup_{x \in \X}\int_Eh(x,w)[\mu-L_n](dw).
\end{align*}
The first term converges at the right rate, thanks to Theorem \ref{th:valueconvergence}, and it remains  to check that 
\[
\limsup_{n\rightarrow\infty}n^{q-1}\mu^n\left(\sup_{x \in \X}\int_Eh(x,w)[\mu-L_n](dw) \ge \epsilon\right) < \infty.
\]
The map $(x,\nu) \mapsto \int_Eh(x,w)\nu(dw)$ is continuous on $\X \times \P_\psi(E)$, and so the map
\[
\P_\psi(E) \ni \nu \mapsto \sup_{x \in \X}\int_Eh(x,w)[\mu-\nu](dw)
\]
is continuous by Berge's theorem \cite[Theorem 17.31]{aliprantisborder}.
Hence, the set 
\[
B := \left\{\nu \in \P_\psi(E) : \sup_{x \in \X}\int_Eh(x,w)[\mu-\nu](dw) \ge \epsilon\right\}
\]
is closed in $\P_\psi(E)$. Theorem \ref{th:lpsanov} then implies
\begin{align*}
\limsup_{n\rightarrow\infty}n^{q-1}\mu^n\left(\sup_{x \in \X}\int_Eh(x,w)[\mu-L_n](dw) \ge \epsilon\right) &\le \left(\inf_{\nu \in B}\alpha(\nu)\right)^{-q},
\end{align*}
where $\alpha$ is defined as in \eqref{def:lpentropy}.
Finally, Lemma \ref{le:alphastrictlypositive} implies that $\inf_{\nu \in B}\alpha(\nu) > 0$.
\end{proof}

Under the assumption $\int_E\psi^q\,d\mu < \infty$, we see that the value $V(L_n)$ always converges to $V(\mu)$ with the polynomial rate $n^{1-q}$. To see when Theorem \ref{th:optimizerconvergence} applies, notice that in many situations, $\X$ is a convex subset of a normed vector space, and  we have uniform convexity in the following form: There exists a strictly increasing function $\varphi$ such that $\varphi(0)=0$ and, for all $t \in (0,1)$ and $x,y \in \X$,
\begin{align*}
\int_E&h(tx + (1-t)y,w)\mu(dw)  \\
	&\le t\int_Eh(x,w)\mu(dw) + (1-t)\int_Eh(y,w)\mu(dw) - t(1-t)\varphi(\|x-y\|).
\end{align*}
See \cite[pp. 202-203]{kaniovski-king-wets} for more on this.

\section{Uniform large deviations and martingales} \label{se:uniform}

This section returns to the example of Section \ref{se:intro:uniform}. 
We first record a useful abstract theorem of F\"ollmer and Schied \cite{follmer-schied-book}, which will allow us to verify tightness of the sub-level sets of $\alpha$ \emph{before knowing it is convex} by checking a property of $\rho$:

\begin{proposition}[Proposition 4.30 of \cite{follmer-schied-book}]  \label{pr:tight}
Suppose a functional $\rho : B(E) \rightarrow \R$ admits the representation
\[
\rho(f) = \sup_{\nu \in \P(E)}\left(\int_Ef\,d\nu - \alpha(\nu)\right), \text{ for } f \in C_b(E),
\]
for some functional $\alpha : \P(E) \rightarrow (-\infty,\infty]$.
Suppose also that there is a sequence $(K_n)$ of compact subsets of $E$ such that 
\[
\lim_{n\rightarrow\infty}\rho(\lambda 1_{K_n}) = \rho(\lambda), \ \forall \lambda \ge 1.
\]
Then $\alpha$ has tight sub-level sets.
\end{proposition}

Fix a convex weakly compact family of probability measures $M \subset \P(E)$. 
Define
\begin{align}
\alpha(\nu) =  \inf_{\mu \in M}H(\nu | \mu), \label{def:uniformrelativeentropy}
\end{align}
where the relative entropy was defined in \eqref{def:relativeentropy}. In light of the classical formula \cite[Proposition 1.4.2]{dupuis-ellis}
\[
\sup_{\nu \in B(E)}\left(\int_Ef\,d\nu - H(\nu | \mu)\right) = \log \int_E e^f \,d\mu,
\]
the $\rho$ corresponding to the functional $\alpha$ given by \eqref{def:uniformrelativeentropy} is then
\begin{align}
\rho(f) &:= \sup_{\nu \in B(E)}\left(\int_Ef\,d\nu - \alpha(\nu)\right) = \sup_{\nu \in B(E)}\sup_{\mu \in M}\left(\int_Ef\,d\nu - H(\nu | \mu)\right) \nonumber \\
	&= \sup_{\mu \in M}\log\int_Ee^f\,d\mu. \label{def:uniformrelativeentropy-rho}
\end{align}
Let us also take note of the famous Donsker-Varadhan formula \cite[Lemma 1.4.3]{dupuis-ellis}
\begin{align}
H(\nu | \mu) &= \sup_{f \in C_b(E)}\left( \int_Ef\,d\nu - \log\int_Ee^f\,d\mu\right). \label{def:donsker-varadhan} 
\end{align}

\begin{lemma} \label{le:uniform-standingassumptions}
The functional $\alpha$ defined in \eqref{def:uniformrelativeentropy} satisfies the standing assumptions. That is, it is convex and bounded from below, and its sub-level sets are weakly compact.
\end{lemma}
\begin{proof}
Note first that $-\log\int_Ee^f\,d\mu$ is convex and weakly continuous in $\mu$ as well as concave and sup-norm continuous in $f$. Thus, using  \eqref{def:donsker-varadhan} and Sion's minimax theorem \cite{sion1958general}, we find
\begin{align*}
\alpha(\nu) &= \inf_{\mu \in M}\sup_{f \in C_b(E)}\left(\int_Ef\,d\nu - \log\int_Ee^f\,d\mu\right) \\
	&= \sup_{f \in C_b(E)}\inf_{\mu \in  M}\left(\int_Ef\,d\nu - \log\int_Ee^f\,d\mu\right) \\
	&= \sup_{f \in C_b(E)}\left(\int_Ef\,d\nu - \rho(f)\right).
\end{align*}
This shows that $\alpha$ is convex and lower semicontinuous. It remains to prove that $\alpha$ has tight sub-level sets, which will follow from Proposition \ref{pr:tight} once we check the second assumption therein. By Prokhorov's theorem, there exist compact sets $K_1 \subset K_2 \subset \cdots$ such that $\sup_{\mu \in M}\mu(K_n^c) \le 1/n$.
Then, for $\lambda \ge 0$, using the formula for $\rho$ of \eqref{def:uniformrelativeentropy-rho},
\begin{align*}
\lambda \ge \rho(\lambda 1_{K_n}) &= \sup_{ \mu \in M}\log\int_E\exp(\lambda 1_{K_n})\,d\mu \\
	&= \sup_{ \mu \in M}\log\left[(e^\lambda - 1)\mu(K_n) + 1\right] \\
	&\ge \log\left[(e^\lambda - 1)(1-1/n) + 1\right].
\end{align*}
As $n\rightarrow\infty$, the right-hand side converges to $\lambda$, which shows $\rho(\lambda 1_{K_n})\rightarrow \lambda = \rho(\lambda)$.
\end{proof}

To compute $\rho_n$, recall that for $M \subset \P(E)$ we define $M_n$ as the set of $\mu \in \P(E^n)$ satisfying $\mu_{0,1} \in M$ and $\mu_{k-1,k}(x_1,\ldots,x_{k-1}) \in M$ for all $k=2,\ldots,n$ and $x_1,\ldots,x_{n-1} \in E$. (Recall that the conditional measures $\mu_{k-1,k}$ were defined in the introduction.) Notice that $M_1=M$.

\begin{proposition} \label{pr:robustentropic-rhon}
For each $n \ge 1$, $\alpha_n(\nu) = \inf_{\mu \in M_n}H(\nu | \mu)$. Moreover,  for each measurable $f :E^n \rightarrow \R \cup \{-\infty\}$ satisfying $\int_{E^n}e^f\,d\mu < \infty$ for every $\mu \in M_n$,
\begin{align}
\rho_n(f) = \sup_{\mu \in M_n}\log\int_{E^n}e^f\,d\mu. \label{def:uniform-rhon}
\end{align}
\end{proposition}
\begin{proof}
Given the first claim, the second follows from the well-known duality
\[
\sup_{\nu \in \P(E^n)}\left(\int_{E^n}f\,d\nu - H(\nu | \mu)\right) = \log\int_{E^n}e^f\,d\mu,
\]
which holds for $\mu \in \P(E^n)$ as long as $e^f$ is $\mu$-integrable (see, e.g., the proof of \cite[1.4.2]{dupuis-ellis}). Indeed, this implies
\begin{align*}
\rho_n(f) &= \sup_{\nu \in \P(E^n)}\left(\int_{E^n}f\,d\nu - \alpha_n(\nu)\right) = \sup_{\mu \in M_n}\sup_{\nu \in \P(E^n)}\left(\int_{E^n}f\,d\nu - H(\nu | \mu)\right) \\
	&= \sup_{\mu \in M_n}\log\int_{E^n}e^f\,d\mu.
\end{align*}
To prove the first claim, note that by definition
\begin{align*}
\alpha_n(\nu) &= \sum_{k=1}^n \int_{E^n}\inf_{\mu \in M}H(\nu_{k-1,k}(x_1,\ldots,x_{k-1}) \, | \, \mu) \nu(dx_1,\ldots,dx_n).
\end{align*}
For $k=2,\ldots,n$ let $\Y_k$ denote the set of measurable maps from $E^{k-1}$ to $M$, and let $\Y_1 = M$. Then the usual measurable selection argument \cite[Proposition 7.50]{bertsekasshreve} yields
\begin{align*}
\alpha_n(\nu) &= \sum_{k=1}^n \inf_{\eta_k \in \Y_k}\int_{E^n}H(\nu_{k-1,k}(x_1,\ldots,x_{k-1}) \, | \, \eta_k(x_1,\ldots,x_{k-1})) \nu(dx_1,\ldots,dx_n).
\end{align*}
Now, if $(\eta_1,\ldots,\eta_n) \in \prod_{k=1}^n\Y_k$, then the measure
\[
\mu(dx_1,\ldots,dx_n) = \eta_1(dx_1)\prod_{k=2}^n\eta_2(x_1,\ldots,x_{k-1})(dx_k)
\]
is in $M$, and $\mu_{k-1,k} = \eta_k$ is a version of the conditional law. Thus
\begin{align*}
\alpha_n(\nu) &\ge \inf_{\mu \in M}\sum_{k=1}^n\int_{E^n}H(\nu_{k-1,k}(x_1,\ldots,x_{k-1}) \, | \, \mu_{k-1,k}(x_1,\ldots,x_{k-1})) \nu(dx_1,\ldots,dx_n).
\end{align*}
On the other hand, for every $\mu \in M_n$, the vector $(\mu_{0,1},\mu_{1,2},\ldots,\mu_{n-1,n})$ belongs to $\prod_{k=1}^n\Y_k$, and we deduce the opposite inequality. Hence
\begin{align*}
\alpha_n(\nu) &\ge \inf_{\mu \in M}\sum_{k=1}^n\int_{E^n}H(\nu_{k-1,k}(x_1,\ldots,x_{k-1}) \, | \, \mu_{k-1,k}(x_1,\ldots,x_{k-1})) \nu(dx_1,\ldots,dx_n) \\
	&= \inf_{\mu \in M}H(\nu|\mu),
\end{align*}
where the last equality follows from the chain rule for relative entropy \cite[Theorem B.2.1]{dupuis-ellis}.
\end{proof}

Theorem \ref{th:main-sanov-extended} now leads to the following uniform large deviation bound:

\begin{corollary} \label{co:uniform-sanov}
For $F \in C_b(\P(E))$, we have
\begin{align*}
\lim_{n\rightarrow\infty}\sup_{\mu \in M_n}\frac{1}{n}\log\int_{E^n}e^{nF \circ L_n}\,d\mu &= \sup_{\nu \in \P(E), \ \mu \in M}\left(F(\nu) - H(\nu | \mu)\right).
\end{align*}
For closed sets $A \subset \P(E)$, we have
\begin{align*}
\lim_{n\rightarrow\infty}\sup_{\mu \in M_n}\frac{1}{n}\log\mu(L_n \in A) &\le -\inf\left\{H(\nu | \mu) : \nu \in A, \ \mu \in M\right\}.
\end{align*}
\end{corollary}
\begin{proof}
The first claim is an immediate consequence of Theorem \ref{th:main-sanov-extended} and the calculation of $\rho_n$ in Proposition \ref{pr:robustentropic-rhon}. To prove the second claim, define $F$ on $\P(E)$ by
\begin{align*}
F(\nu) = \begin{cases}
0 &\text{if } \nu \in A, \\
-\infty &\text{otherwise.}
\end{cases}
\end{align*}
Then $F$ is upper semicontinuous and bounded from above. Use Proposition \ref{pr:robustentropic-rhon} to compute
\begin{align*}
\rho_n(nF \circ L_n) &= \sup_{\mu \in M_n}\log\int_{E^n}\exp(nF \circ L_n)\,d\mu = \sup_{\mu \in M_n}\log\mu(L_n \in A).
\end{align*}
The proof is completed by applying Theorem \ref{th:main-sanov-extended} with this function $F$.
\end{proof}

The following proposition simplifies the strong Cram\'er condition \eqref{def:strongcramer-condition} in the present context.

\begin{proposition} \label{pr:robustcramercondition}
Let $\psi : E \rightarrow \R_+$ be measurable. Suppose that for every $\lambda > 0$ we have
\begin{align}
\sup_{\mu \in M}\int_E e^{\lambda\psi}\,d\mu < \infty. \label{ass:pr:robustcramercondition}
\end{align}
Then the strong Cram\'er condition holds, i.e., $\lim_{m\rightarrow\infty}\rho(\lambda\psi 1_{\{\psi \ge m\}})\rightarrow 0$ for all $\lambda > 0$. In particular, the sub-level sets of $\alpha$ are pre-compact subsets of $\P_\psi(E)$.
\end{proposition}
\begin{proof}
Because $e^{\lambda\psi}$ is $\mu$-integrable for each $\mu\in M$ and $\lambda > 0$, Proposition \ref{pr:robustentropic-rhon} implies
\begin{align*}
\rho(\lambda\psi 1_{\{\psi \ge m\}}) &= \sup_{\mu \in M}\log\int_E\exp\left(\lambda\psi 1_{\{\psi \ge m\}}\right)\,d\mu \\
	&\le \sup_{\mu \in M}\log\left( 1 + \int_{\{\psi \ge m\}}e^{\lambda\psi}\,d\mu\right).
\end{align*}
Now note that $1_{\{\psi \ge m\}} \le \frac{\psi}{m} \le \frac{1}{m}e^\psi$ pointwise, and thus the assumption \eqref{ass:pr:robustcramercondition} yields 
\begin{align*}
\lim_{m\to\infty}\sup_{\mu \in M}\int_{\{\psi \ge m\}}e^{\lambda\psi}\,d\mu &\le \lim_{m\to\infty}\sup_{\mu \in M}\frac{1}{m}\int_E e^{(1+\lambda)\psi}\,d\mu = 0.
\end{align*}
\end{proof}

We are finally ready to specialize Corollary \ref{co:uniform-sanov} to prove Theorem \ref{th:azuma}, similarly to how we specialized Theorem \ref{th:lpsanov} to prove Corollary \ref{co:lpcramer} in Section \ref{se:nonexpLDP}.

\subsection*{Proof of Theorem \ref{th:azuma}}
Define
\[
M = \left\{ \mu \in \P(\R^d) : \log\int_{\R^d} e^{\langle y,x\rangle}\mu(dx) \le \varphi(y), \ \forall y \in \R^d\right\}.
\]
We claim that $M$ is weakly compact. Indeed, it is clearly convex, and closedness follows from Fatou's lemma (cf.\ \cite[Theorem A.3.12]{dupuis-ellis}). To prove tightness, let $e_1,\ldots,e_d$ denote the standard basis vectors in $\R^d$. Write $x=(x_1,\ldots,x_d)$ for a generic element of $\R^d$. For each $\mu \in M$ and  $t > 0$, Markov's inequality yields
\begin{align*}
\mu\left\{x\in \R^d :\max_{i=1,\ldots,d}|x_i| >t\right\} &\le \sum_{k=1}^d\left(\mu\{x \in \R^d : x_i > t/2\} +  \mu\{x \in \R^d : -x_i > t/2\}\right) \\
	&\le \sum_{k=1}^de^{-t/2}\int_{\R^d} (e^{x_i} + e^{-x_i})\,\mu(dx) \\
	&\le e^{-t/2}\sum_{k=1}^de^{\varphi(e_i) + \varphi(-e_i)},
\end{align*}
and we deduce that $M$ is tight.
Now define $\psi(x) = \sum_{i=1}^d|x_i|$ and notice that 
\begin{align*}
\sup_{\mu \in M}\int_{\R^d}\exp(\lambda\psi)\,d\mu < \infty, \text{ for all } \lambda \ge 0.
\end{align*}
Proposition \ref{pr:robustcramercondition} then shows that the strong Cram\'er condition holds.
Define a closed set $B \subset \P_\psi(\R^d)$ by $B = \{\nu \in \P_\psi(\R^d) : \int_{\R^d} z\,\nu(dz) \in A\}$, where $A$ was the given closed subset of $\R^d$. Corollary \ref{co:uniform-sanov} yields
\begin{align*}
\limsup_{n\rightarrow\infty}\sup_{\mu \in M_n}\frac{1}{n}\log\mu(L_n \in B) &\le -\inf\left\{\alpha(\nu) : \nu \in \P_\psi(\R^d), \ \int x\,\nu(dx) \in A\right\},
\end{align*}
Now let $(S_0,\ldots,S_n) \in \mathcal{S}_{d,\varphi}$.
The law of $S_1$ belongs to $M$, and the conditional law of $S_k-S_{k-1}$ given $S_1,\ldots,S_{k-1}$ belongs almost surely to $M$, for each $k$, and so the law of $(S_1,S_2-S_1,\ldots,S_n-S_{n-1})$ belongs to $M_n$. Thus
\[
\PP\left(S_n/n \in A\right) \le \sup_{\mu \in M_n}\mu(L_n \in B),
\]
and all that remains is to prove that 
\begin{align*}
\inf\left\{\alpha(\nu) : \nu \in \P_\psi(\R^d), \ \int z\,\nu(dz) \in A\right\} \ge \inf_{x \in A}\varphi^*(x).
\end{align*}
To prove this, it suffices to show $\Psi(x) \ge \varphi^*(x)$ for every $x \in \R^d$, where
\begin{align}
\Psi(x) := \inf\left\{\alpha(\nu) : \nu \in \P_\psi(\R^d), \ \int z\,\nu(dz) =x\right\}. \label{pf:azuma1}
\end{align}
To this end, note that for all $y \in \R^d$
\[
\rho(\langle \cdot,y\rangle) = \sup_{\mu \in M}\log\int_Ee^{\langle z,y\rangle}\mu(dz) \le \varphi(y),
\]
and then use the representation of Proposition \ref{pr:cramer-representation} to get
\begin{align*}
\Psi(x) &= \sup_{y \in \R^d}\left(\langle x,y\rangle - \rho(\langle \cdot,y\rangle)\right) \ge \sup_{y \in \R^d}\left(\langle x,y\rangle - \varphi(y)\right) = \varphi^*(x).
\end{align*}
{\ } \hfill\qedsymbol

\section{Optimal transport and control} \label{se:optimaltransport}

This section discusses the example of Section \ref{se:intro:optimaltransport} in more detail.
Again let $E$ be a Polish space, and fix a lower semicontinuous function $c : E^2 \rightarrow [0,\infty]$ which is not identically equal to $\infty$. Fix $\mu \in \P(E)$, and define
\[
\alpha(\nu) = \inf_{\pi \in \Pi(\mu,\nu)}\int c\,d\pi,
\]
where $\Pi(\mu,\nu)$ is the set of probability measures on $E \times E$ with first marginal $\mu$ and second marginal $\nu$. Assume that $\int_Ec(x,x)\mu(dx) < \infty$; in many practical cases, $c(x,x)=0$ for all $x$, so this is not a restrictive assumption and merely ensures that $\alpha(\mu) < \infty$. Kantorovich duality \cite[Theorem 1.3]{villani-book} shows that
\begin{align*}
\alpha(\nu) &= \sup\left(\int_Ef\,d\nu - \int_Eg\,d\mu : f,g \in C_b(E), \ f(y) - g(x) \le c(x,y) \ \forall x,y\right).
\end{align*}
This immediately shows that $\alpha$ is convex and weakly lower semicontinuous.
The next two lemmas identify, respectively, the dual $\rho$ and the modest conditions that ensure that $\alpha$ has compact sub-level sets.

\begin{lemma}
Given $\alpha$ as above, and defining $\rho$ as usual by \eqref{intro:duality}, we have
\begin{align}
\rho(f) = \int_ER_cf\,d\mu, \text{ for all } f \in B(E), \label{def:rho-optimaltransport}
\end{align}
where $R_cf : E \rightarrow \R$ is defined by
\[
R_cf(x) = \sup_{y \in E}\left(f(y) - c(x,y)\right).
\]
\end{lemma}
\begin{proof}
Note that $R_cf$ is universally measurable (e.g., by \cite[Proposition 7.50]{bertsekasshreve}), so the integral in \eqref{def:rho-optimaltransport} makes sense. Now compute
\begin{align*}
\rho(f) &= \sup_{\nu \in \P(E)}\left(\int_Ef\,d\nu - \alpha(\nu)\right) \\
	&= \sup_{\nu \in \P(E)}\sup_{\pi \in \Pi(\mu,\nu)}\left(\int_Ef\,d\nu - \int_{E^2}c\,d\pi\right) \\
	&= \sup_{\pi \in \Pi(\mu)}\int_{E^2}\left(f(y) - c(x,y)\right)\pi(dx,dy),
\end{align*}
where $\Pi(\mu)$ is the set of $\pi \in \P(E \times E)$ with first marginal $\mu$.
Use the standard measurable selection theorem \cite[Proposition 7.50]{bertsekasshreve} to find a measurable map $Y : E \rightarrow E$ such that $R_cf(x) = f(Y(x)) - c(x,Y(x))$ for $\mu$-a.e. $x$. Then, choosing $\pi(dx,dy) = \mu(dx)\delta_{Y(x)}(dy)$ shows
\[
\rho(f) \ge \int_E\left(f(Y(x))-c(x,Y(x))\right)\mu(dx) = \int_ER_cf\,d\mu.
\]
On the other hand, it is clear that for every $\pi \in \Pi(\mu)$ we have
\[
\int_{E^2}\left(f(y) - c(x,y)\right)\pi(dx,dy) \le \int_{E}\sup_{y \in E}\left(f(y) - c(x,y)\right)\mu(dx) = \int_ER_cf\,d\mu.
\]
\end{proof}

\begin{lemma} \label{le:tightnessfunction}
Suppose that for each compact set $K \subset E$, the function $h_K(y) := \inf_{x \in K}c(x,y)$ has pre-compact sub-level sets.\footnote{In fact, since $c$ is lower semicontinuous, so is $h_K$ (see \cite[Lemma 17.30]{aliprantisborder}). Thus, our assumption is equivalent to requiring $\{y \in E : h_K(y) \le m\}$ to be compact for each $m \ge 0$.}  Then $\alpha$ has compact sub-level sets.
\end{lemma}
\begin{proof}
We already know that $\alpha$ has closed sub-level sets, so we must show only that they are tight.
Fix $\nu \in \P(E)$ such that $\alpha(\nu) < \infty$ (noting that such $\nu$ certainly exist, as $\mu$ is one example). Fix $\epsilon > 0$, and find $\pi \in \Pi(\mu,\nu)$ such that
\begin{align}
\int c\,d\pi \le \alpha(\nu) + \epsilon < \infty. \label{pf:tightnessfunction1}
\end{align}
As finite measures on Polish spaces are tight, we may find a compact set $K \subset E$ such that $\mu(K^c) \le \epsilon$.
Set $K_n := \{y \in E : h_K(y) < n\}$ for each $n$, and note that this set is pre-compact by assumption. Disintegrate $\pi$ by finding a measurable map $E \ni x \mapsto \pi_x \in \P(E)$ such that $\pi(dx,dy) = \mu(dx)\pi_x(dy)$. By Markov's inequality, for each $n > 0$ and each $x \in K$ we have
\begin{align*}
\pi_x(K_n^c) &\le \pi_x\{y \in E : c(x,y) \ge n\} \le \frac{1}{n}\int_E c(x,y)\pi_x(dy).
\end{align*}
Using this and the inequality \eqref{pf:tightnessfunction1} along with the assumption that $c$ is nonnegative,
\begin{align*}
\nu(K_n^c) &= \int_E\mu(dx)\pi_x(K_n^c) \\
	&\le \mu(K^c) + \int_K\mu(dx)\pi_x(K_n^c) \\
	&\le \epsilon + \frac{1}{n}\int_K\mu(dx)\int_E\pi_x(dy) c(x,y) \\
	&\le \epsilon + \frac{1}{n}\int_{E \times E}c\,d\pi \\
	&\le \left(1 + \frac{1}{n}\right)\epsilon + \frac{1}{n}\alpha(\nu).
\end{align*}
As $\epsilon$ was arbitrary, we have $\nu(K_n^c) \le \alpha(\nu)/n$. Thus, for each $m > 0$, the sub-level set $\{\nu \in \P(E) : \alpha(\nu) \le m\}$ is contained in the tight set
\[
\bigcap_{n=1}^\infty\left\{\nu \in \P(E) : \nu(K_n^c) \le m/n\right\}.
\]
\end{proof}

Let us now compute $\rho_n$. It is convenient to work with more probabilistic notation, so let us suppose $(X_i)_{i=1}^\infty$ is a sequence of i.i.d.\ $E$-valued random variables with common law $\mu$, defined on some fixed probability space. For each $n$, let $\Y_n$ denote the set of equivalence classes of a.s. equal $E^n$-valued random variables $(Y_1,\ldots,Y_n)$ where $Y_k$ is $(X_1,\ldots,X_k)$-measurable for each $k=1,\ldots,n$.

\begin{proposition} \label{pr:rhon-optimaltransport}
For each $n \ge 1$ and each $f \in B(E)$,
\[
\rho_n(f) = \sup_{(Y_1,\ldots,Y_n) \in \Y_n}\E\left[f(Y_1,\ldots,Y_n) - \sum_{i=1}^nc(X_i,Y_i)\right].
\]
\end{proposition}
\begin{proof}
The proof is by induction. Let us first rewrite $\rho$ in our probabilistic notation:
\[
\rho(f) = \E\left[\sup_{y \in E}[f(y)-c(X_1,y)]\right].
\]
Using a standard measurable selection argument \cite[Proposition 7.50]{bertsekasshreve}, we deduce
\[
\rho(f) = \sup_{Y_1 \in \Y_1}\E\left[f(Y_1)-c(X_1,Y_1)\right]
\]
The inductive step proceeds as follows. Suppose we have proven the claim for a given $n$. Fix $f \in B(E^{n+1})$ and define $g \in B(E^n)$ by
\[
g(x_1,\ldots,x_n) := \rho(f(x_1,\ldots,x_n,\cdot)),
\]
so that by Proposition \ref{pr:rhon-iterative} we have $\rho_{n+1}(f)=\rho_n(g)$.
Since $X_1$ and $X_{n+1}$ have the same distribution, we may relabel to find
\begin{align*}
g(x_1,\ldots,x_n) &= \sup_{Y_1 \in \Y_1}\E\left[f(x_1,\ldots,x_n,Y_1)-c(X_1,Y_1)\right] \\
	&= \sup_{Y_{n+1} \in \Y_{n+1}^1}\E\left[f(x_1,\ldots,x_n,Y_{n+1})-c(X_{n+1},Y_{n+1})\right],
\end{align*}
where we define $\Y^1_{n+1}$ to be the set of $X_{n+1}$-measurable $E$-valued random variables. Now note that any $(Y_1,\ldots,Y_n)$ in $\Y_n$ is $(X_1,\ldots,X_n)$-measurable, and independence of $(X_i)_{i=1}^\infty$ implies
\[
g(Y_1,\ldots,Y_n) = \sup_{Y_{n+1} \in \Y_{n+1}^1}\E\left[\left.f(Y_1,\ldots,Y_n,Y_{n+1})-c(X_{n+1},Y_{n+1})\right| Y_1,\ldots,Y_n\right].
\]
We claim that
\begin{align}
\E\left[g(Y_1,\ldots,Y_n)\right] = \sup_{Y_{n+1}}\E\left[f(Y_1,\ldots,Y_n,Y_{n+1}) - c(X_{n+1},Y_{n+1})\right], \label{pf:jointly-measurable0}
\end{align}
where the supremum is over $(X_1,\ldots,X_{n+1})$-measurable $E$-valued random variables $Y_{n+1}$. Indeed, once this is established, we conclude as desired (using Proposition \ref{pr:rhon-iterative}) that
\begin{align*}
\rho_{n+1}(f) &= \rho_n(g) = \sup_{(Y_1,\ldots,Y_n) \in \Y_n}\E\left[g(Y_1,\ldots,Y_n) - \sum_{i=1}^nc(X_i,Y_i)\right] \\
	&= \sup_{(Y_1,\ldots,Y_n) \in \Y_n}\sup_{Y_{n+1}}\E\left[f(Y_1,\ldots,Y_n,Y_{n+1}) - \sum_{i=1}^{n+1}c(X_i,Y_i)\right].
\end{align*}
Hence, the rest of the proof is devoted to justifying \eqref{pf:jointly-measurable0}, which is really an interchange of supremum and expectation.

Note that $\Y_{n+1}^1$ is a Polish space when topologized by convergence in measure. The function $h : E^n \times \Y^1_{n+1} \rightarrow \R$ given by
\[
h(x_1,\ldots,x_n;Y_{n+1}) := \E\left[f(x_1,\ldots,x_n,Y_{n+1})-c(X_{n+1},Y_{n+1})\right].
\]
is jointly measurable. Note as before that independence implies that for every $(Y_1,\ldots,Y_n) \in \Y_n$ and $Y_{n+1} \in \Y^1_{n+1}$ we have, for a.e. $\omega$,
\begin{align}
h(Y_1(\omega),&\ldots,Y_n(\omega);Y_{n+1}) = \E\left[\left. f(Y_1,\ldots,Y_n,Y_{n+1})-c(X_{n+1},Y_{n+1})\right| Y_1,\ldots,Y_n\right](\omega). \label{pf:jointly-measurable1}
\end{align}
Using the usual measurable selection theorem \cite[Proposition 7.50]{bertsekasshreve} we get
\begin{align*}
\E\left[g(Y_1,\ldots,Y_n)\right] &= \E\left[\sup_{Y_{n+1} \in \Y_{n+1}^1}h(Y_1(\cdot),\ldots,Y_n(\cdot);Y_{n+1})\right] \\
	&= \sup_{H \in \widetilde{\Y}_{n+1}^1}\E\left[h(Y_1(\cdot),\ldots,Y_n(\cdot);H(Y_1,\ldots,Y_n))\right],
\end{align*}
where $\widetilde{\Y}_{n+1}^1$ denotes the set of measurable maps $H : E^n \rightarrow \Y^1_{n+1}$. But a measurable map $H : E^n \rightarrow \Y^1_{n+1}$ can be identified almost everywhere with an $(X_1,\ldots,X_{n+1})$-measurable random variable $Y_{n+1}$. Precisely, by Lemma \ref{le:jointmeasurability} (in the appendix) there exists a jointly measurable map $\varphi : E^{n+1} \rightarrow E$ such that, for $\mu^n$-a.e. $(x_1,\ldots,x_n) \in E^n$, we have
\[
\varphi(x_1,\ldots,x_{n+1}) = H(x_1,\ldots,x_n)(x_{n+1}), \text{ for } \mu\text{-a.e. } x_{n+1} \in E.
\]
Define $Y_{n+1} = \varphi(X_1,\ldots,X_{n+1})$, and note that \eqref{pf:jointly-measurable1} implies, for a.e. $\omega$,
\begin{align*}
h(Y_1(\omega),&\ldots,Y_n(\omega);H(Y_1,\ldots,Y_n)) = \E\left[\left.f(Y_1,\ldots,Y_n,Y_{n+1})-c(X_{n+1},Y_{n+1})\right| Y_1,\ldots,Y_n\right](\omega).
\end{align*}
This identification of $\widetilde{\Y}_{n+1}^1$ and the tower property of conditional expectations leads to \eqref{pf:jointly-measurable0}.
\end{proof}

\appendix

\section{A recursive formula for $\rho_n$}  \label{se:rhon}

In this brief section we make rigorous the claim in \eqref{def:intro:rhon-recursive}. To do so requires a brief review of analytic sets, needed only for this section. A subset of a Polish space is \emph{analytic} if it is the image of a Borel subset of another Polish space through a Borel measurable function. A real-valued function $f$ on a Polish space is \emph{upper semianalytic} if $\{f \ge c\}$ is an analytic set for each $c \in \R$. It is well known that every analytic set is universally measurable \cite[Corollary 7.42.1]{bertsekasshreve}, and thus every upper semianalytic function is universally measurable. The defining formula for $\rho_n$ given in \eqref{intro:duality} makes sense even when $f : E^n \rightarrow \R$ is bounded and universally measurable, or in particular when $f$ is upper semianalytic.

\begin{proposition} \label{pr:rhon-iterative}
Let $n > 1$. Suppose $f : E^n \rightarrow \R$ is upper semianalytic. Define $g : E^{n-1} \rightarrow \R$ by
\[
g(x_1,\ldots,x_{n-1}) = \rho\left(f(x_1,\ldots,x_{n-1},\cdot)\right).
\]
Then $g$ is upper semianalytic, and $\rho_n(f) = \rho_{n-1}(g)$.
\end{proposition}
\begin{proof}
To show that $g$ is upper semianalytic, note that
\begin{align*}
g(x_1,\ldots,x_{n-1}) &= \rho(f(x_1,\ldots,x_{n-1},\cdot)) \\
	&= \sup_{\nu \in \P(E)}\left(\int_Ef(x_1,\ldots,x_{n-1},\cdot)\,d\nu - \alpha(\nu)\right).
\end{align*}
Clearly $\alpha$ is Borel measurable, as its sub-level sets are compact.
It follows from \cite[Proposition 7.48]{bertsekasshreve} that the term in parentheses is upper semianalytic as a function of $(x_1,\ldots,x_{n-1},\nu)$. Hence, $g$ is itself upper semianalytic, by \cite[Proposition 7.47]{bertsekasshreve}.

We now turn toward the proof of the recursive formula for $\rho_n$.
Note first that the definition of $\alpha_n$ can be written recursively by setting $\alpha_1 = \alpha$ and, for $\nu \in \P(E^n)$ and a kernel $K$ from $E^n$ to $E$ (i.e., a Borel measurable map $x \mapsto K_x$ from $E^n$ to $\P(E)$), setting
\begin{align}
\alpha_{n+1}\left(\nu(dx)K_x(dx_{n+1})\right) = \int_{E^n}\alpha(K_x)\nu(dx) + \alpha_n(\nu). \label{pf:rhon-iterative1}
\end{align}
Fix $f \in B(E^{n+1})$, and note that $g(x_1,\ldots,x_n) := \rho(f(x_1,\ldots,x_n,\cdot))$ is upper semianalytic by the above argument. By definition,
\begin{align}
\rho_n(g) &= \sup_{\nu \in \P(E^n)}\left\{\int_{E^n}g\,d\nu - \alpha_n(\nu)\right\}. \label{pf:rhon-iterative2}
\end{align}
By a well known measurable selection argument \cite[Proposition 7.50]{bertsekasshreve}, for each $\nu \in \P(E^n)$ it holds that
\begin{align*}
\int_{E^n}g\,d\nu &= \int_{E^n}\sup_{\eta \in \P(E)}\left(\int_Ef(x_1,\ldots,x_n,x_{n+1})\eta(dx_{n+1}) - \alpha(\eta) \right)\nu(dx) \\
	&= \sup_{K}\left(\int_{E^n}\int_Ef(x_1,\ldots,x_{n+1})K_{x}(dx_{n+1})\nu(dx) - \int_{E^n}\alpha(K_x) \nu(dx)\right),
\end{align*}
where we have abbreviated $x=(x_1,\ldots,x_n)$, and where the supremum is over all kernels from $E^n$ to $E$, i.e., all Borel measurable maps from $E^n$ to $\P(E)$.\footnote{A priori, the supremum should be taken over maps $K$ from $E^n$ to $\P(E)$ which are measurable with respect to the smallest $\sigma$-field containing the analytic sets. But any such map is universally measurable and thus agrees $\nu$-a.e. with a Borel measurable map.} Every probability measure on $E^{n+1}$ can be written as $\nu(dx)K_x(dx_{n+1})$ for some $\nu \in \P(E^n)$ and some kernel $K$ from $E^n$ to $E$. Thus, in light of \eqref{pf:rhon-iterative1} and \eqref{pf:rhon-iterative2},
\begin{align*}
\rho_n(g) &= \sup_{\nu \in \P(E^n)}\sup_{K}\left[\int_{E^n}\int_Ef(x_1,\ldots,x_{n+1})K_{x}(dx_{n+1})\nu(dx) - \int_{E^n}\alpha(K_x) \nu(dx) - \alpha_n(\nu)\right] \\
	&= \sup_{\nu \in \P(E^{n+1})}\left(\int_{E^n}f\,d\nu - \alpha_{n+1}(\nu)\right) \\
	&= \rho_n(f).
\end{align*}
\end{proof}

In general, the function $g$ in Proposition \ref{pr:rhon-iterative} can fail to be Borel measurable. For instance, if $E$ is compact and $\alpha \equiv 0$, then our standing assumptions hold. In this case $\rho(f) = \sup_{x \in E}f(x)$ for $f \in B(E)$. For $f \in B(E^2)$ we have $\rho(f(x,\cdot)) = \sup_{y \in E}f(x,y)$. If $f(x,y)=1_A(x,y)$ for a Borel set $A \subset E^2$ whose projections are not Borel, then $\rho(f(x,\cdot))$ is not Borel. Credit is due to Daniel Bartl for pointing out this simple counterexample to an inaccurate claim in an earlier version of the paper; his paper \cite{bartl2016pointwise} shows why semianalytic functions are essential in this context.

\section{Two technical lemmas} \label{se:technical-lemma}

Here we state and prove a technical lemma that was used in the proof of Proposition \ref{pr:rhon-optimaltransport} as well as a simple extension of Jensen's inequality to convex functions of random measures.
The first lemma essentially says that if $f=f(x,y)$ is a function of two variables such that the map $x \mapsto f(x,\cdot)$ is measurable, from $E$ into an appropriate function space, then $f$ is essentially jointly measurable:

\begin{lemma} \label{le:jointmeasurability}
Let $(\Omega,\F,P)$ be a standard Borel probability space, let $E$ be a Polish space, and let $\mu \in \P(E)$. Let $L^0$ denote the set of equivalence classes of $\mu$-a.e. equal measurable functions from $E$ to $E$, and endow $L^0$ with the topology of convergence in measure. If $H : \Omega \rightarrow L^0$ is measurable, then there exists a jointly measurable function $h : \Omega \times E \rightarrow E$ such that, for $P$-a.e. $\omega$, we have $H(\omega)(x) = h(\omega,x)$ for $\mu$-a.e. $x \in E$.
\end{lemma}
\begin{proof}
By Borel isomorphism, we may assume without loss of generality that $\Omega = E = [0,1]$. In particular, $H(\omega)(x) \in [0,1]$ for all $\omega,x \in [0,1]$. Let $L^1$ denote the set of $P\times\mu$-integrable (equivalence classes of a.s. equal) measurable functions from $[0,1]^2$ to $\R$. Define a linear functional $T : L^1 \rightarrow \R$ by
\[
T(\varphi) = \int P(d\omega)\int\mu(dx)H(\omega)(x)\varphi(\omega,x).
\]
This is well-defined because the function 
\[
\omega \mapsto \int\mu(dx)H(\omega)(x)\varphi(\omega,x)
\]
is measurable; indeed, this is easily checked for $\varphi$ of the form $\varphi(\omega,x)=f(\omega)g(x)$, for $f$ and $g$ bounded and measurable, and the general case follows from a monotone class argument. Because $|H(\omega)(x)|\le 1$, it is readily checked that $T$ is continuous. Thus $T$ belongs to the continuous dual of $L^1$, and there exists a bounded measurable function $h : [0,1]^2 \rightarrow \R$ such that
\[
T(\varphi) = \int P(d\omega)\int\mu(dx)h(\omega,x)\varphi(\omega,x),
\]
for all $\varphi \in L^1$. It is straightforward to check that this $h$ has the desired property.
\end{proof}

Our final lemma, an infinite-dimensional form of Jensen's inequality, is surely known, but we were unable to locate a precise reference, and the proof is quite short.

\begin{lemma} \label{le:jensen}
Fix $P \in \P(\P(E))$, and define the mean measure $\overline{P} \in \P(E)$ by
\[
\overline{P}(A) = \int_{\P(E)}m(A)\,P(dm).
\]
Then, for any function $G : \P(E) \rightarrow (-\infty,\infty]$ which is convex, bounded from below, and weakly lower semicontinuous, we have
\[
G(\overline{P}) \le \int_{\P(E)}G \, dP.
\]
\end{lemma}
\begin{proof}
Define (on some probability space) i.i.d.\ $\P(E)$-valued random variables $(\mu_i)_{i=1}^\infty$ with common law $P$. Define the partial averages
\[
S_n = \frac{1}{n}\sum_{i=1}^n\mu_i.
\]
For any $f \in C_b(E)$,  the law of large numbers implies
\[
\lim_{n\rightarrow \infty}\int_Ef\,dS_n = \lim_{n\rightarrow \infty}\frac{1}{n}\sum_{i=1}^n\int_Ef\,d\mu_i = \E\int_Ef\,d\mu_1 = \int_Ef\,d\overline{P}, \ \ a.s.
\]
This easily shows that $S_n \rightarrow \overline{P}$ weakly a.s.
Use Fatou's lemma and the assumptions on $G$ to get
\begin{align*}
G(\overline{P}) &\le \liminf_{n\rightarrow \infty}\E[G(S_n)] \le \liminf_{n\rightarrow \infty}\E\left[\frac{1}{n}\sum_{i=1}^nG(\mu_i)\right] = \E[G(\mu_1)] = \int_{\P(E)}G \, dP.
\end{align*}
\end{proof}

\subsection*{Acknowledgements}
The author is indebted to Stephan Eckstein and Daniel Bartl as well as two anonymous referees for their careful feedback, which greatly improved the exposition and accuracy of the paper.

\bibliographystyle{amsplain}
\bibliography{sanov-bib}

\providecommand{\bysame}{\leavevmode\hbox to3em{\hrulefill}\thinspace}
\providecommand{\MR}{\relax\ifhmode\unskip\space\fi MR }
\providecommand{\MRhref}[2]{%
  \href{http://www.ams.org/mathscinet-getitem?mr=#1}{#2}
}
\providecommand{\href}[2]{#2}
\begin{thebibliography}{10}

\bibitem{acciaio-penner-dynamic}
B.~Acciaio and I.~Penner, \emph{Dynamic risk measures}, Advanced Mathematical
  Methods for Finance, Springer, 2011, pp.~1--34.

\bibitem{agueh2011barycenters}
M.~Agueh and G.~Carlier, \emph{Barycenters in the {W}asserstein space}, SIAM
  Journal on Mathematical Analysis \textbf{43} (2011), no.~2, 904--924.

\bibitem{aliprantisborder}
C.~Aliprantis and K.~Border, \emph{Infinite dimensional analysis: {A}
  hitchhiker's guide}, 3 ed., Springer, 2007.

\bibitem{atar2015robust}
R.~Atar, K.~Chowdhary, and P.~Dupuis, \emph{Robust bounds on risk-sensitive
  functionals via {R}{\'e}nyi divergence}, SIAM/ASA Journal on Uncertainty
  Quantification \textbf{3} (2015), no.~1, 18--33.

\bibitem{backhoff2018non}
J.~{Backhoff-Veraguas}, D.~Lacker, and L.~Tangpi, \emph{Non-exponential {S}anov
  and {S}childer theorems on {W}iener space: {BSDE}s, {S}chr\"{o}dinger
  problems and control}, arXiv preprint arXiv:1810.01980 (2018).

\bibitem{bartl2016pointwise}
D.~Bartl, \emph{Conditional nonlinear expectations}, arXiv preprint
  arXiv:1612.09103v2 (2017).

\bibitem{bental-teboulle-1986}
A.~Ben-Tal and M.~Teboulle, \emph{Expected utility, penalty functions, and
  duality in stochastic nonlinear programming}, Management Science \textbf{32}
  (1986), no.~11, 1445--1466.

\bibitem{bental-teboulle-2007}
\bysame, \emph{An old-new concept of convex risk measures: {T}he optimized
  certainty equivalent}, Mathematical Finance \textbf{17} (2007), no.~3,
  449--476.

\bibitem{bertsekasshreve}
D.~Bertsekas and S.~Shreve, \emph{Stochastic optimal control: The discrete time
  case}, Athena Scientific, 1996.

\bibitem{blanchet2015optimal}
A.~Blanchet and G.~Carlier, \emph{Optimal transport and {C}ournot-{N}ash
  equilibria}, Mathematics of Operations Research \textbf{41} (2015), no.~1,
  125--145.

\bibitem{bobkov-ding}
S.~Bobkov and Y.~Ding, \emph{Optimal transport and {R}{\'e}nyi informational
  divergence}, Preprint (2014).

\bibitem{boissard2011simple}
E.~Boissard, \emph{Simple bounds for the convergence of empirical and
  occupation measures in 1-{W}asserstein distance}, Electronic Journal of
  Probability \textbf{16} (2011), 2296--2333.

\bibitem{borovkov}
A.A. Borovkov and K.A. Borovkov, \emph{Asymptotic analysis of random walks:
  {H}eavy-tailed distributions}, no. 118, Cambridge University Press, 2008.

\bibitem{cheridito2011composition}
P.~Cheridito and M.~Kupper, \emph{Composition of time-consistent dynamic
  monetary risk measures in discrete time}, International Journal of
  Theoretical and Applied Finance \textbf{14} (2011), no.~01, 137--162.

\bibitem{de2001convergence}
L.~de~Haan and T.~Lin, \emph{On convergence toward an extreme value
  distribution in ${C}[0, 1]$}, Annals of probability (2001), 467--483.

\bibitem{dembozeitouni}
A.~Dembo and O.~Zeitouni, \emph{Large deviations techniques and applications},
  vol.~38, Springer Science \& Business Media, 2009.

\bibitem{denisov2008large}
D.~Denisov, A.B. Dieker, and V.~Shneer, \emph{Large deviations for random walks
  under subexponentiality: the big-jump domain}, The Annals of Probability
  \textbf{36} (2008), no.~5, 1946--1991.

\bibitem{ding2014wasserstein}
Y.~Ding, \emph{Wasserstein-{D}ivergence transportation inequalities and
  polynomial concentration inequalities}, Statistics \& Probability Letters
  \textbf{94} (2014), 77--85.

\bibitem{dudley1969speed}
R.M. Dudley, \emph{The speed of mean {G}livenko-{C}antelli convergence}, The
  Annals of Mathematical Statistics \textbf{40} (1969), no.~1, 40--50.

\bibitem{dudley2018real}
\bysame, \emph{Real analysis and probability}, CRC Press, 2018.

\bibitem{dupacova1988asymptotic}
J.~Dupacov{\'a} and R.J.B. Wets, \emph{Asymptotic behavior of statistical
  estimators and of optimal solutions of stochastic optimization problems}, The
  annals of statistics (1988), 1517--1549.

\bibitem{dupuis-ellis}
P.~Dupuis and R.S. Ellis, \emph{A weak convergence approach to the theory of
  large deviations}, vol. 902, John Wiley \& Sons, 2011.

\bibitem{eckstein2017extended}
S.~Eckstein, \emph{Extended {L}aplace principle for empirical measures of a
  {M}arkov chain}, arXiv preprint arXiv:1709.02278 (2017).

\bibitem{eichelsbacher1996large}
P.~Eichelsbacher and U.~Schmock, \emph{Large deviations of products of
  empirical measures and {U}-{E}mpirical measures in strong topologies}, Univ.
  Bielefeld, Sonderforschungsbereich 343, Diskrete Strukturen in der Math.,
  1996.

\bibitem{einmahl2008characterization}
U.~Einmahl and D.~Li, \emph{Characterization of {LIL} behavior in {B}anach
  space}, Transactions of the American Mathematical Society \textbf{360}
  (2008), no.~12, 6677--6693.

\bibitem{follmer-knispel}
H.~F{\"o}llmer and T.~Knispel, \emph{Entropic risk measures: {C}oherence vs.
  convexity, model ambiguity and robust large deviations}, Stochastics and
  Dynamics \textbf{11} (2011), no.~02n03, 333--351.

\bibitem{follmer-schied-convex}
H.~F{\"o}llmer and A.~Schied, \emph{Convex measures of risk and trading
  constraints}, Finance and stochastics \textbf{6} (2002), no.~4, 429--447.

\bibitem{follmer-schied-book}
\bysame, \emph{Stochastic finance: {A}n introduction in discrete time}, Walter
  de Gruyter, 2011.

\bibitem{foss2011introduction}
S.~Foss, D.~Korshunov, and S.~Zachary, \emph{An introduction to heavy-tailed
  and subexponential distributions}, vol.~6, Springer, 2011.

\bibitem{fournier2015rate}
N.~Fournier and A.~Guillin, \emph{On the rate of convergence in {W}asserstein
  distance of the empirical measure}, Probability Theory and Related Fields
  \textbf{162} (2015), no.~3-4, 707--738.

\bibitem{fuqing2012relative}
G.~Fuqing and X.~Mingzhou, \emph{Relative entropy and large deviations under
  sublinear expectations}, Acta Mathematica Scientia \textbf{32} (2012), no.~5,
  1826--1834.

\bibitem{hardy-littlewood-polya}
G.H. Hardy, J.E. Littlewood, and G.~P{\'o}lya, \emph{Inequalities}, Cambridge
  university press, 1952.

\bibitem{hu2010cramer}
F.~Hu, \emph{On {C}ram{\'e}r's theorem for capacities}, Comptes Rendus
  Mathematique \textbf{348} (2010), no.~17, 1009--1013.

\bibitem{hult2006regular}
H.~Hult and F.~Lindskog, \emph{Regular variation for measures on metric
  spaces}, Publ. Inst. Math.(Beograd)(NS) \textbf{80} (2006), no.~94, 121--140.

\bibitem{hult2005functional}
H.~Hult, F.~Lindskog, T.~Mikosch, and G.~Samorodnitsky, \emph{Functional large
  deviations for multivariate regularly varying random walks}, The Annals of
  Applied Probability \textbf{15} (2005), no.~4, 2651--2680.

\bibitem{kall1987approximations}
P.~Kall, \emph{On approximations and stability in stochastic programming},
  Parametric Optimization and Related Topics, Akademie-Verlag, 1987,
  pp.~387--407.

\bibitem{kaniovski-king-wets}
Y.M. Kaniovski, A.J. King, and R.J.B. Wets, \emph{Probabilistic bounds (via
  large deviations) for the solutions of stochastic programming problems},
  Annals of Operations Research \textbf{56} (1995), no.~1, 189--208.

\bibitem{king1991epi}
A.J. King and R.J.B. Wets, \emph{Epi-consistency of convex stochastic
  programs}, Stochastics and Stochastic Reports \textbf{34} (1991), no.~1-2,
  83--92.

\bibitem{lacker-liquidity}
D.~Lacker, \emph{Liquidity, risk measures, and concentration of measure}, arXiv
  preprint arXiv:1510.07033 (2015).

\bibitem{lindskog2014regularly}
F.~Lindskog, S.I. Resnick, and J.~Roy, \emph{Regularly varying measures on
  metric spaces: {H}idden regular variation and hidden jumps}, Probability
  Surveys \textbf{11} (2014), 270--314.

\bibitem{mikosch-nagaev}
T.~Mikosch and A.V. Nagaev, \emph{Large deviations of heavy-tailed sums with
  applications in insurance}, Extremes \textbf{1} (1998), no.~1, 81--110.

\bibitem{mogul1977large}
A.A. Mogul'skii, \emph{Large deviations for trajectories of multi-dimensional
  random walks}, Theory of Probability \& Its Applications \textbf{21} (1977),
  no.~2, 300--315.

\bibitem{nagaev1979large}
S.V. Nagaev, \emph{Large deviations of sums of independent random variables},
  The Annals of Probability (1979), 745--789.

\bibitem{owari2014maximum}
K.~Owari, \emph{Maximum {L}ebesgue extension of monotone convex functions},
  Journal of Functional Analysis \textbf{266} (2014), no.~6, 3572--3611.

\bibitem{parthasarathy2005probability}
K.R. Parthasarathy, \emph{Probability measures on metric spaces}, vol. 352,
  American Mathematical Soc., 2005.

\bibitem{peng2007law}
S.~Peng, \emph{Law of large numbers and central limit theorem under nonlinear
  expectations}, arXiv preprint math/0702358 (2007).

\bibitem{peng2010nonlinear}
\bysame, \emph{Nonlinear expectations and stochastic calculus under
  uncertainty}, arXiv preprint arXiv:1002.4546 (2010).

\bibitem{petrov2012sums}
V.~Petrov, \emph{Sums of independent random variables}, vol.~82, Springer
  Science \& Business Media, 2012.

\bibitem{rhee2016sample}
C.-H. Rhee, J.~Blanchet, and B.~Zwart, \emph{Sample path large deviations for
  heavy-tailed {L}{\'e}vy processes and random walks}, arXiv preprint
  arXiv:1606.02795 (2016).

\bibitem{schied1998cramer}
A.~Schied, \emph{Cramer's condition and {S}anov's theorem}, Statistics \&
  probability letters \textbf{39} (1998), no.~1, 55--60.

\bibitem{sion1958general}
M.~Sion, \emph{On general minimax theorems}, Pacific J. Math \textbf{8} (1958),
  no.~1, 171--176.

\bibitem{villani-book}
C.~Villani, \emph{Topics in optimal transportation}, no.~58, American
  Mathematical Soc., 2003.

\bibitem{wang2010sanov}
R.~Wang, X.~Wang, and L.~Wu, \emph{Sanov's theorem in the {W}asserstein
  distance: {A} necessary and sufficient condition}, Statistics \& Probability
  Letters \textbf{80} (2010), no.~5, 505--512.

\bibitem{weed2017sharp}
J.~Weed and F.~Bach, \emph{Sharp asymptotic and finite-sample rates of
  convergence of empirical measures in {W}asserstein distance}, arXiv preprint
  arXiv:1707.00087 (2017).

\bibitem{zalinescu-book}
C.~Zalinescu, \emph{Convex analysis in general vector spaces}, World
  Scientific, 2002.

\end{thebibliography}

\end{document}